\documentclass[a4paper,10pt,reqno]{smfart}
\usepackage{etex}
\usepackage{amssymb,amsmath,mathrsfs,graphicx,mathptm,float,mathdots}
\usepackage[T1]{fontenc}
\usepackage[english,francais]{babel}
\usepackage{easybmat}

\theoremstyle{plain}
\newtheorem{thm}{Theorem}[section]
\newtheorem{pro}[thm]{Proposition}
\newtheorem{lem}[thm]{Lemma}

\newtheorem{theoreme}{Theorem}

\newtheorem*{conj}{Conjecture}

\theoremstyle{definition}
\newtheorem{defi}[thm]{Definition}

\newtheorem{eg}[thm]{Example}
\newtheorem{egs}[thm]{Examples}
\newtheorem{rem}[thm]{Remark}

\def\og{\leavevmode\raise.3ex\hbox{$\scriptscriptstyle\langle\!\langle$~}}
\def\fg{\leavevmode\raise.3ex\hbox{~$\!\scriptscriptstyle\,\rangle\!\rangle$}}

\newcommand{\transp}[1]{{\vphantom{#1}}^{\mathit t}{#1}}

\addtocounter{section}{0}             
\numberwithin{equation}{section}       

\begin{document}
\selectlanguage{english}

\author{Julie D\'eserti}

\address{Institut de Math\'ematiques de Jussieu, Universit\'e Paris $7,$ Projet G\'eom\'etrie et Dynamique, Site Chevaleret, Case $7012,$ $75205$ Paris Cedex 13, France.}
\email{deserti@math.jussieu.fr}

\author{Julien Grivaux}

\address{Institut de Math\'ematiques de Jussieu, Universit\'e Paris $6,$ Projet Topologie et G\'eom\'etrie Alg\'ebrique, Site Chevaleret, Case $7012,$ $75205$ Paris Cedex 13, France.}
\email{jgrivaux@math.jussieu.fr}

\title{Automorphisms of rational surfaces with positive topological entropy}
\maketitle

\begin{abstract}
A complex compact surface which carries an automorphism of positive topological entropy has been proved by Cantat to be either a torus, a K$3$ surface, an Enriques surface or a rational surface. Automorphisms of rational surfaces are quite mysterious and have been recently the object of intensive studies. In this paper, we construct several new examples of automorphisms of rational surfaces with positive topological entropy. We also explain how to define and to count parameters in families of birational maps of $\mathbb{P}^2(\mathbb{C})$ and in families of rational surfaces.

\end{abstract}

\begin{footnotesize}
\tableofcontents
\end{footnotesize}

\bigskip\bigskip\bigskip\bigskip

\section*{Introduction}

\noindent If $X$ is a topological space and $f$ is a homeomorphism of $X,$ the topological entropy of $f,$ denoted by $\mathrm{h}_{\text{top}}(f),$ is a nonnegative number measuring the complexity of the dynamical system $(X,f).$ If $X$ is a compact K\"{a}hler manifold and~$f$ is a biholomorphism of $X,$ then $\mathrm{h}_{\text{top}}(f)=\sup_{1\leq p\leq \dim X}\delta_p(f),$ where $\delta_p(f)$ is the $p$-th dynamical degree of $f,$ {\it i.e.} the spectral radius of $f^*$ acting on $\mathrm{H}^{p,p}(X)$ (\emph{see} \cite{Gr1, Gr2, Yo}). When $X$ is a complex compact surface (K\"{a}hler or not) carrying a biholomorphism of positive topological entropy, Cantat has proved \cite{C} that $X$ is either a complex torus, a K$3$ surface, an Enriques surface or a nonminimal rational surface. Although automorphisms of complex tori are easy to describe, it is rather difficult to construct automorphisms on K$3$ surfaces or rational surfaces (constructions and dynamical properties of automorphisms of K$3$ surfaces can be found in \cite{C} and \cite{Mc2}).

 \medskip

\noindent The first examples of rational surfaces endowed with biholomorphisms of positive entropy are due to Kummer and Coble \cite{Co}. The Coble surfaces are obtained by blowing up the ten nodes of a nodal sextic in $\mathbb{P}^2(\mathbb{C})$ and the Kummer surfaces are desingularizations of quotients of complex $2$-tori by involutions with fixed points.
Obstructions to the existence of such biholomorphisms on rational surfaces are also known: if $X$ is a rational surface and $f$ is a biholomorphism of $X$ such that $f$ has positive topological entropy, then the representation of the automorphism group of $X$ in $\mathrm{GL}(\mathrm{Pic}(X))$ given by $g\mapsto g^*$ has infinite image. This implies by a result of Harbourne \cite{Harbourne} that its kernel is finite, so that $X$ has no nonzero holomorphic vector field. A second consequence which follows from \cite[Th. \!5]{Nag} is that $X$ is \textit{basic}, {\it i.e.} can be obtained by successive blowups from the projective plane $\mathbb{P}^2(\mathbb{C});$ furthermore, the number of blowups must be at least ten.

\medskip

\noindent The first infinite families of examples have been constructed independently in \cite{MC} and \cite{BK2} by different methods; the rational surfaces are obtained by blowing up distinct points of $\mathbb{P}^2(\mathbb{C}).$ The corresponding automorphisms come from birational quadratic maps of $\mathbb{P}^2(\mathbb{C})$ which are of the form $A\sigma,$ where $A$ is in $\mathrm{PGL}(3;\mathbb{C})$ and $\sigma$ is the Cremona involution. These constructions yield a countable family of examples.

\medskip

\noindent More recently, Bedford and Kim constructed arbitrary big holomorphic families of rational surfaces endowed with biholomorphisms of positive entropy. These families are explicitly given as follows:

\begin{theoreme}[\cite{BK3}]\label{thmbk}
{\sl Consider two integers $n\geq 3$ and $k\geq 2$ such that $n$ is odd and $(n,k)\not=(3,2).$ There exists a nonempty subset $C_k$ of $\mathbb{R}$ such that, if~$c\in C_k$ and $a=(a_2,a_4,\ldots,a_{n-3})\in \mathbb{C}^{\frac{n-3}{2}},$ the map
\begin{equation}\label{bedfordkk}
f_a\colon(x:y:z)\to\big(xz^{n-1}:z^n:x^n-yz^{n-1}+cz^n+\sum_{\stackrel{\ell=2}{\ell \text{ even}}}^{n-3}a_\ell x^{\ell+1}z^{n-\ell-1}\big)
\end{equation}

\noindent can be lifted to an automorphism of positive topological entropy of a rational surface $X_a.$ The surfaces $X_a$ are obtained by blowing up $k$ infinitely near points of length $2n-1$ on the invariant line $\{x=0\}$ and form a holomorphic family over the parameter space given by the $a_j\,'s.$ If $k=2$ and $n\geq 5$ is odd, then there exists a neighborhood of $0$ in $\mathbb{C}^{\frac{n-3}{2}}$ such that for all distinct elements $a$ and $a'$ in $U$ with $a_{n-3}\neq 0,$ $X_a$ and $X_{a'}$ are not biholomorphic.}
\end{theoreme}

\noindent These examples are generalizations of the birational cubic map introduced by \cite{HV, HV2} and studied by \cite{Ta1, Ta2, Ta3}.

\medskip

\noindent The present paper has two distinct aims: the first one is to give a general procedure to construct examples of rational surfaces carrying biholomorphisms of positive entropy in a more systematic way than what has been done before. The second one is to associate with any holomorphic family of automorphisms of rational surfaces a number, called the generic number of parameters, and to give a geometrical interpretation of this number using deformation theory.

\medskip

\noindent Our strategy for the construction of automorphisms of rational surfaces is the following one: we start by choosing any birational map of the complex projective plane $f.$ By the standard factorization theorem for birational maps on surfaces as a composition of blow up and blow down \cite[IV \S 3.4]{Sh}, there exist two canonical sets of (possibly infinitely near) points $\widehat{\xi}_1$ and~$\widehat{\xi}_2$ in~$\mathbb{P}^2(\mathbb{C})$ such that $f$ can be lifted to an isomorphism between $\mathrm{Bl}_{\widehat{\xi}_1}\mathbb{P}^2$ and $\mathrm{Bl}_{\widehat{\xi}_2}\mathbb{P}^2,$ where~$\mathrm{Bl}_{\widehat{\xi}_j}\mathbb{P}^2$  denotes the rational surfaces obtained by blowing up $\mathbb{P}^2(\mathbb{C})$ at the points of  $\widehat{\xi}_j.$ The data of $\widehat{\xi}_1$ and $\widehat{\xi}_2$ allow to get automorphisms of rational surfaces in the left $\mathrm{PGL}(3;\mathbb{C})$-orbit of $f:$ for a fixed positive integer~$k,$ let $\varphi$ be an element of~$\mathrm{PGL}(3;\mathbb{C})$ such that $\widehat{\xi}_1,$ $\varphi\widehat{\xi}_2,$ $(\varphi f)\varphi\widehat{\xi}_2,$ $\ldots,$ $(\varphi f)^{k-1}\varphi\widehat{\xi}_2$ have pairwise disjoint supports in $\mathbb{P}^2(\mathbb{C})$ and that $(\varphi f)^k\varphi\widehat{\xi}_2=\widehat{\xi}_1.$ Then~$\varphi f$ can be lifted to an automorphism of $\mathbb{P}^2(\mathbb{C})$ blown up at $\widehat{\xi}_1,$ $\varphi\widehat{\xi}_2,$ $(\varphi f)\varphi\widehat{\xi}_2,$ $\ldots,$ $(\varphi f)^{k-1}\varphi\widehat{\xi}_2.$ Furthermore, if the conditions above are satisfied for a holomorphic family of $\varphi,$ we get a holomorphic family of rational surfaces whose dimension is at most eight.
Therefore, we see that the problem of lifting an element in the $\mathrm{PGL}(3;\mathbb{C})$-orbit of~$f$ to an automorphism is strongly related to the equation $u(\widehat{\xi}_2)=\widehat{\xi}_1,$ where~$u$ is a germ of biholomorphism of $\mathbb{P}^2(\mathbb{C})$ mapping the support of $\widehat{\xi}_2$ to the support of $\widehat{\xi}_1.$ In concrete examples, when $\widehat{\xi}_1$ and~$\widehat{\xi}_2$ are known, this equation can actually be solved and reduces to polynomial equations in the Taylor expansions of $u$ at the various points of the support of $\widehat{\xi}_2.$ It is worth pointing out that in the generic case, $\widehat{\xi}_1$ and $\widehat{\xi}_2$ consist of the same number~$d$ of distinct points in the projective plane, and the equa\-tion~$u(\widehat{\xi}_2)=\widehat{\xi}_1$ gives $2d$ independent conditions on $u$ (which is the maximum possible number if~$\widehat{\xi}_1$ and $\widehat{\xi}_2$ have length $d$). Conversely, infinitely near points can considerably decrease the number of conditions on $u$ as shown in our examples. This explains why holomorphic families of automorphisms of rational surfaces are more likely to occur when multiple blowups are made.

\medskip

\noindent Let us now describe in more details the examples we obtain. We do not deal with the case of the Cremona involution~$\sigma$ because birational maps of the type $A\sigma,$ with $A$ in $\mathrm{PGL}(3;\mathbb{C}),$ are linear fractional recurrences studied in \cite{BK1, BK2}, and our approach does not give anything new in this case. Our first examples proceed from a family $(\Phi_n)_{n\geq 2}$ of birational maps of $\mathbb{P}^2(\mathbb{C})$ given by $\Phi_n(x:y:z)=(xz^{n-1}+y^n:yz^{n-1}:z^n).$ These birational maps are very special because their exceptional locus is a single line, the line $\{z=0\};$ and their locus of indeterminacy is a single point on this line, the point $P=(1\!:\!0\!:0).$ We compute explicitly $\widehat{\xi}_1$ and $\widehat{\xi}_2$ and obtain sequences of blowups already done in \cite{BK3}. Then we exhibit various families of solutions of the equation $(\varphi \, \Phi_n)^{k-1}\varphi \,\widehat{\xi}_2=\widehat{\xi}_1$ for $k=2,3$ and $\varphi$ in $\mathrm{PGL}( 3;\mathbb{C}),$ and obtain in this way automorphisms of rational surfaces with positive topological entropy. Many of our examples are similar or even sometimes linearly conjugated to those constructed in \cite{BK3}: they have an invariant line and the sequences of blowups are of the same type.
However, for $(n,k)=(3,2)$ we have found an example of a different kind:
\begin{theoreme}\label{onefam}
{\sl If $\alpha$ is a complex number in $\mathbb{C}\setminus\{0,\,1\},$ let $\varphi_\alpha$ be the element of $\mathrm{PGL(3;\mathbb{C})}$ given by
\begin{align*}
&\varphi_\alpha=\left[\begin{array}{ccc}\alpha & 2(1-\alpha) & 2+\alpha-\alpha^2\\
-1 & 0 & \alpha+1 \\
1 & -2 & 1-\alpha
\end{array} \right].
\end{align*}

\noindent The map $\varphi_\alpha\Phi_3$ has no invariant line and is conjugate to an automorphism of $\mathbb{P}^2(\mathbb{C})$ blown up in $15$ points; its first dynamical degree is $\frac{3+\sqrt{5}}{2}.$ Besides, the family $\varphi_\alpha\Phi_3$ is holomorphically trivial.}
\end{theoreme}

\medskip


\noindent The meaning of "holomorphically trivial" in the statement of the theorem is that two generic transformations in the family are linearly conjugate, {\it i.e.} they are conjugate via an automorphism of the complex projective plane. The family of Theorem \ref{onefam}, as well as our other examples, are all holomorphically trivial. After many attempts to produce nontrivial holomorphic families, we have been led to conjecture that holomorphic families of the type $\varphi_\alpha\Phi_n$ which yield automorphisms of positive topological entropy are all holomorphically trivial. Nevertheless this phenomenon is not a generality and seems specific to the maps $\Phi_n$: indeed Theorem \ref{thmbk} gives for $n\geq 5$ examples of families of birational maps of the type $\varphi_\alpha f$ conjugate to families of automorphisms of positive entropy on rational surfaces which are not holomorphically trivial (\emph{see} \S \ref{conjecture}).

\medskip

\noindent We finally carry out our method for another birational cubic map,  namely $f(x:y:z)=\big(y^2 z : x(xz+y^2) : y(xz+y^2)\big).$ This map blows down a conic and a line intersecting transversally the conic along the two points of indeterminacy. We produce automorphisms of rational surfaces with positive topological entropy in the left $\mathrm{PGL}(3; \mathbb{C})$-orbit of $f$:
\begin{theoreme}
{\sl Let $\alpha$ be a nonzero complex number and
\begin{align*}
&\varphi_\alpha=\left[\begin{array}{ccc}\dfrac{2\alpha^3}{343}(37\mathrm{i}\sqrt{3}+3)& \alpha& -\dfrac{2\alpha^{2}}{49}(5\mathrm{i}\sqrt{3}+11)\\[2ex]
\dfrac{\alpha^2}{49}(-15+11\mathrm{i}\sqrt{3}) & 1 & -\dfrac{\alpha}{14}(5\mathrm{i}\sqrt{3}+11)\\[2ex]
 -\dfrac{\alpha}{7}(2\mathrm{i}\sqrt{3}+3)& 0 & 0\end{array} \right]
\end{align*}

\smallskip

\noindent The map $\varphi_\alpha f$ is conjugate to an automorphism of $\mathbb{P}^2(\mathbb{C})$ blown up in $15$ points, its first dynamical degree is $\frac{3+\sqrt{5}}{2}.$ Besides, the family $\varphi_\alpha f$ is holomorphically trivial.}
\end{theoreme}
\noindent This example seems completely new, the configuration of exceptional curves shows that it is not linearly conjugate to any of the already known examples (although we do not know if it is the case when linear conjugacy is replaced by birational conjugacy).

\par \bigskip

\noindent We now turn to the second object of the paper, which is the count of parameters in a family of automorphism of rational surfaces, or even more generally to a family of birational maps of $\mathbb{P}^2(\mathbb{C}),$ and its geometric interpretation. The naive way to count the number of parameters in a family of birational maps would be to count the number of parameters appearing in the homogeneous polynomials of the family. This is not a very good idea: indeed, if $f$ is any birational map, the adjoint orbit of $f$ under the action of $\mathrm{PGL}(3; \mathbb{C})$ is a family of birational maps which are all linearly conjugate. In other words, the parameters appearing in this family are fake parameters from the point of view of complex dynamics. This is the reason why the authors study in Theorem \ref{thmbk} how the complex structures of the rational surfaces $X_a$ vary with $a.$ Their approach links implicitly two different objects: a family of rational surface automorphisms, considered simply as a family of birational maps; and a family of rational surfaces.

\par \medskip

\noindent In the paper, we define a notion of generic number of parameters for each of these two geometric objects. First we associate a number with any holomorphic family of birational maps of the projective plane, which takes account of the possible fake parameters; we call it the generic number of parameters of the family. By definition, holomorphically trivial families will be families with vanishing generic number of parameters and generically effective families will be families with maximal generic number of parameters. The principal advantage of this number is that it is quite easy to compute in concrete examples. Then we propose an approach in order to define and count the generic number of parameters in a given family of rational surfaces. The good setting for this study is the theory of deformations of complex compact manifolds of Kodaira and Spencer \cite{Kodaira}: a deformation is a triplet $(\mathfrak{X},\pi,B)$ such that $\mathfrak{X}$ and $B$ are complex manifolds and $\pi\colon\mathfrak{X}\to B$ is a proper holomorphic submersion. If $X$ is a complex compact manifold, a deformation of $X$ is a deformation $(\mathfrak{X},\pi,B)$ such that for a specific $b$ in $B,$ $\mathfrak{X}_b$ is biholomorphic to $X.$ If $(\mathfrak{X},\pi,B)$ is a deformation and $X$ is a fiber of $\mathfrak{X},$ Ehresmann's fibration theorem implies that $\mathfrak{X}$ is diffeomorphic to~$X\times B$ over $B.$ Therefore, a deformation can also be seen as a family of integrable complex structures $(J_b)_{b\in B}$ on a fixed differentiable manifold $X,$ varying holomorphically with $b.$ The main tool of deformation theory is the Kodaira-Spencer map: if~$(\mathfrak{X},\pi,B)$ is a deformation and $b_{_{0}}$ is a point of $B,$ the Kodaira-Spencer map of $\mathfrak{X}$ at $b_{_{0}}$ is a linear map $\mathrm{KS}_{b_{_0}}(\mathfrak{X}) \colon \mathrm{T}_{b_{_{0}}}B\to\mathrm{H}^1(\mathfrak{X}_{b_{_{0}}},\mathrm{T}\mathfrak{X}_{b_{_{0}}}),$ which is intuitively the differential of the map $b\mapsto J_{b}$ at $b_{_{0}}.$ If $(\mathfrak{X},\pi,B)$ is a deformation whose fibers are projective varieties, we prove that the kernels of $\mathrm{KS}_b(\mathfrak{X})$ have generically the same dimension and define a holomorphic subbundle~$E_\mathfrak{X}$ of $\mathrm{T}B.$ This leads to a natural definition of the generic number $\mathfrak{m}(\mathfrak{X})$ of parameters of a deformation $(\mathfrak{X},\pi,B)$ as $\mathfrak{m}(\mathfrak{X})=\dim B-\mathrm{rank}\,E_\mathfrak{X}.$ In the case $\mathfrak{m}(\mathfrak{X})=\dim B,$ we say that $\mathfrak{X}$ is generically effective; in other words, for $b$ generic in $B,$ $\mathrm{KS}_b (\mathfrak{X})$ is injective. 
This is slightly weaker than requiring that different generic fibers of $\mathfrak{X}$ are not biholomorphic (as in Theorem \ref{thmbk}), but much easier to verify in concrete examples.

\medskip

\noindent From a theoretical point of view, deformations of basic rational surfaces are easy to understand. If we fix a positive integer $N,$ the moduli space of sets of ordered (possibly infinitely near) points in the projective plane $\mathbb{P}^2(\mathbb{C})$ of cardinal~$N$ is a smooth projective variety $S_N$ of dimension $2N$ obtained by blowing up successive incidence loci. Besides, there exists a natural deformation $\mathfrak{X}_N$ over $S_N$ whose fibers are rational surfaces: if $\widehat{\xi}$ is in $S_N,$ then $(\mathfrak{X}_N)_{\widehat{\xi}}$ is equal to $\mathrm{Bl}_{\widehat{\xi}} \, \mathbb{P}^2.$ This deformation is complete at any point of $S_N,$ {\it i.e.} every deformation of a fiber of $\mathfrak{X}_N$ is locally induced by $\mathfrak{X}_N$ up to holomorphic base change. Therefore, if $(\mathfrak{X}, \pi, B)$ is a deformation of a basic rational surface, all the fibers in a small neighborhood of the central fiber remain rational and basic (this is no longer the case for nonbasic rational surfaces). There is a natural $\mathrm{PGL}(3;\mathbb{C})$-action on $S_N$ which can be lifted on $\mathfrak{X}_N,$ and this action can be used to describe the Kodaira-Spencer map of the deformation $\mathfrak{X}_N:$ for any $\widehat{\xi}$ in $S_N,$ $\mathrm{KS}_{\widehat{\xi}}\, (\mathfrak{X}_N)$ is surjective and its kernel is the tangent space at $\widehat{\xi}$ of the $\mathrm{PGL}(3;\mathbb{C})$-orbit of $\widehat{\xi}$ in $S_N.$ If $N \geq 4, $ let $S_N^{\dag}$ be the Zariski-dense open set of points $\widehat{\xi}$ in $S_N$ such that~$\mathrm{Bl}_{\widehat{\xi}} \, \mathbb{P}^2$ has no nonzero holomorphic vector field. Since $\mathfrak{X}_N$ is complete, families of rational surfaces with no nonzero holomorphic vector fields can locally be described as the pullback of $\mathfrak{X}_N$ by a holomorphic map from the parameter space to $S_N^{\dag}.$ We provide a practical way to count the generic number of parameters in such families:

\begin{theoreme}\label{para}
{\sl Let $U$ be an open set in $\mathbb{C}^n,$ $N$ be an integer greater than or equal to $4$ and $\psi\colon U \to S_N^{\dag}$ be a holomorphic map. Then~$\mathfrak{m}(\psi^* \mathfrak{X}_N)$ is the smallest integer $k$ such that for all generic $\alpha$ in $U,$ there exist a neighborhood $\Omega$ of~$0$ in~$\mathbb{C}^{n-k}$ and two holomorphic maps $\gamma\,  \colon \Omega \rightarrow U$ and $M \colon \Omega \rightarrow \mathrm{PGL(3; \mathbb{C})}$ such that:
\begin{itemize}
\item $\gamma_{*}(\, 0)$ is injective,
\item $\gamma\, (\, 0)=\alpha$ and $M(0\!)=\mathrm{Id},$
\item for all $t$ in $\Omega,$ $\psi \, (\gamma (\, t \!))=M(\, t)\,  \psi(\alpha).$
\end{itemize}}
\end{theoreme}

\noindent Remark that we deal only with rational surfaces without nontrivial holomorphic vector field. This hypothesis is not very restrictive in our context because these surfaces are the only ones which carry interesting automorphisms, \textit{i.e.} automorphisms of infinite order when acting on the Picard group of the surface.
\par \medskip

\noindent For every positive integer $d,$ let us introduce the set $\mathrm{Bir}_d(\mathbb{P}^2)$ of birational maps of the complex projective plane given by a triplet of homogeneous polynomials of degree $d$ without common factors. As an application of Theorem \ref{para}, we can compare the two notions of generic number of parameters we have introduced:

\begin{theoreme}\label{ineq}
{\sl Let $N$ and $d$ be positive integers such that $N$ is greater than or equal to $4,$ $Y$ be a smooth connected analytic subset of $\mathrm{Bir}_d(\mathbb{P}^2)$ and $\psi \colon Y \rightarrow S_N^{\dag}$ be a holomorphic map. If $\mathfrak{X}=\psi^{*} \mathfrak{X}_N,$ let $\Gamma \colon \mathfrak{X} \rightarrow Y \times \mathbb{P}^2(\mathbb{C})$ be the natural holomorphic map over $Y$ whose restriction on each fiber $\mathfrak{X}_y$ is  the natural projection from $\mathrm{Bl}_{\psi(y)} \mathbb{P}^2$ to~$\mathbb{P}^2(\mathbb{C}).$ Assume that for any $y$ in $Y,$ if $f_y$ is the birational map parameterized by $y,$ ${\Gamma_y}^{-1} \circ f_y \circ \Gamma_y$ is an automorphism of the rational surface $\mathfrak{X}_y.$
Then the generic number of parameters of the holomorphic family $Y$ is smaller than the generic number of parameters of the deformation $\mathfrak{X},$ {\it i.e.} $\mathfrak{m}(Y)\leq \mathfrak{m}(\mathfrak{X}).$}
\end{theoreme}

\noindent As a corollary, if a family of automorphisms of rational surfaces without holomorphic vector field is generically effective (as a family of birational maps), then the associated family of rational surfaces is generically effectively parameterized.

\par \medskip

\noindent Under the assumptions of the previous theorem, the equality between $\mathfrak{m}(Y)$ and $\mathfrak{m}(\mathfrak{X})$ is not valid in general. However, for families of automorphisms of rational surfaces considered in the article, we prove that the two notions of generic number of parameters agree:

\begin{theoreme}
{\sl Let $k$ and $N$ be two positive integers, $f$ be a birational map of the complex projective plane, $\widehat{\xi}_1$ and $\widehat{\xi}_2$ be two points of $S_N$ corresponding to the minimal desingularization of $f$ and $U$ be a smooth connected analytic subset of $\mathrm{PGL}(3; \mathbb{C}).$ We make the following assumptions:
\begin{itemize}
\item [(i)] For all $\varphi$ in $U,$ $(\varphi f)^k \varphi\,  \widehat{\xi}_2=\widehat{\xi}_1.$
\item [(ii)] The supports of $\widehat{\xi}_1,$ $\varphi \widehat{\xi}_2$ and $(\varphi f)^j \varphi\,  \widehat{\xi}_2, \, \, 1 \leq j \leq k-1,$ are pairwise disjoint.
\item [(iii)] If $\psi \colon U \rightarrow S_{kN}$ is defined by $\psi(\varphi)=( \widehat{\xi}_1, \varphi \, \widehat{\xi}_2, \varphi f  \varphi\,  \widehat{\xi}_2, \ldots, (\varphi f)^{k-1} \varphi\,  \widehat{\xi}_2),$ then the image of $\psi$ is included in $S_{kN}^{\dag}.$
\item[(iv)] For all $\varphi$ in $U,$ the birational map $\varphi f$ can be lifted to an automorphism of the rational surface $\mathrm{Bl}_{\psi(\varphi)}\mathbb{P}^2.$
\end{itemize}
If $\widetilde{U}$ denotes the family of birational maps $(\varphi f)_{\varphi \in U}$ and if $\mathfrak{X}=\psi^*\mathfrak{X}_{kN},$ then $\mathfrak{m}(\widetilde{U})=\mathfrak{m}(\mathfrak{X}).$}
\end{theoreme}

\par \medskip

\noindent\textbf{Acknowledgements. ---} We would like to thank E. Bedford, D. Cerveau and C. Favre for fruitful discussions and M. Manetti for the reference \cite{Ho}. We also thank the referee for his/her very helpful comments.

\bigskip

\section{Algebraic and dynamical properties of birational maps}

\subsection{First dynamical degree}

\noindent A {\it rational map} from $\mathbb{P}^2(\mathbb{C})$ into itself is a map of the following type
\begin{align*}
& f\colon\mathbb{P}^2(\mathbb{C})\dashrightarrow\mathbb{P}^2(\mathbb{C}),&&
(x:y:z)\mapsto\big(f_0(x,y,z):f_1(x,y,z):f_2(x,y,z)\big)
\end{align*}
where the $f_i$'s are homogeneous polynomials of the same degree without common factor. The {\it degree} of $f$ is equal to the degree of the $f_i$'s. A {\it birational map}
is a rational map whose inverse is also rational. The
birational maps of $\mathbb{P}^2(\mathbb{C})$ into itself form a group which is
called the {\it Cremona group} and denoted by $\mathrm{Bir}(\mathbb{P}^2).$ The elements of $\mathrm{Bir}(\mathbb{P}^2)$ are sometimes called {\it Cremona transformations}.

\medskip

\noindent If $f$ is a birational map, $\mathrm{Ind}\,f$ denotes the finite set of points blown up by $f;$ this is the set of the common zeroes of the~$f_i's.$ We say that $\mathrm{Ind}\,f$ is the {\it locus of indeterminacy} of~$f.$ The set of curves collapsed by $f,$ called {\it exceptional locus} of~$f,$ is deno\-ted by $\mathrm{Exc}\,f;$ it can be obtained by computing the jacobian determinant of $f.$

\medskip

\noindent The degree is not a birational invariant; if $f$ and $g$ are in $\mathrm{Bir}(\mathbb{P}^2),$
then usually $\deg (gfg^{-1})$ and $\deg f$ are different. Nevertheless there exist two strictly
positive constants $\alpha$ and $\beta$ such that for all integer $n$ the following holds:
\begin{align*}
\alpha\deg f^n\leq\deg (gf^ng^{-1})\leq\beta\deg f^n.
\end{align*}
\noindent This means that the degree growth is a birational invariant.

\medskip

\noindent Let us recall the
notion of first dynamical degree introduced in \cite{Fr, RS}: if $f$ is in $\mathrm{Bir}(\mathbb{P}^2),$ the {\it first dynamical degree} of  $f$ is defined by
\begin{align*}
\lambda(f)=\lim (\deg f^n)^{1/n}.
\end{align*}

\noindent More generally we can define this notion for bimeromorphic maps of a K\"{a}hler surface. A bimeromorphic map $f$ on a K\"{a}hler surface $X$ induces a map $f^*$ from $\mathrm{H}^{1,1}(X,\mathbb{R})$ into itself. The first dynamical degree of $f$ is given by
\begin{align*}
\lambda(f)=\lim (\vert (f^n)^*\vert)^{1/n}.
\end{align*}

\noindent Let $f$ be a map on a complex compact K\"{a}hler surface; the notions of first dynamical degree and topological entro\-py~$\mathrm{h}_{\text{top}}(f)$ are related by the following formula: $\mathrm{h}_{\text{top}}(f)=\log\lambda(f)$ (\emph{see} \cite{Gr1, Gr2, Yo}).

\noindent Diller and Favre characterize the birational maps of $\mathbb{P}^2(\mathbb{C})$ up to birational conjugacy; the case of automorphisms with quadratic growth is originally due to Gizatullin.

\begin{thm}[\cite{DiFa, Gi}]\label{DillerFavre}
{\sl Let $f$ be a bimeromorphic map of a K\"{a}hler surface. Up to bimeromorphic
conjugacy, one and only one of the following holds.
\begin{itemize}
\item The sequence $(\vert (f^n)^*\vert)_{n\in \mathbb{N}}$ is bounded, $f$ is an automorphism on some
rational surface and an iterate of $f$ is an automorphism isotopic
to the identity.

\item The sequence $(\vert (f^n)^*\vert)_{n\in \mathbb{N}}$ grows linearly and $f$ preserves a rational
fibration; in this case $f$ is not an automorphism.

\item The sequence $(\vert (f^n)^*\vert)_{n\in \mathbb{N}}$ grows quadratically and $f$ is an automorphism
preserving an elliptic fibration.

\item The sequence $(\vert (f^n)^*\vert)_{n\in \mathbb{N}}$ grows exponentially.
\end{itemize}
\medskip

\noindent In the second (resp. third) case, the invariant fibration is unique. In the three first
cases  $\lambda(f)$ is equal to $1,$ in the last case $\lambda(f)$ is strictly larger
than $1.$}
\end{thm}

\begin{egs} Let us give some examples.
\begin{itemize}
\item If $f$ is an automorphism of $\mathbb{P}^2(\mathbb{C})$ or a birational map of finite order, then $(\deg f^n)_n$ is bounded.

\item The map $f=(xy:yz:z^2)$ satisfies that $(\deg f^n)_n$
grows linearly.

\item The map $f_\varepsilon=\big((y+z)(y+z-\varepsilon z):x(y-\varepsilon z):(y+z)z\big)$ grows quadratically as soon as $\varepsilon$ belongs to $\{1/2,1/3\}$ (\emph{see} \cite[Prop. \!9.5]{DiFa}).

\item A H\'enon map, {\it i.e.} an automorphism of $\mathbb{C}^2$ of the form
\begin{align*}
& f=(y,P(y)-\delta x), && \delta\in\mathbb{C}^*,\, P\in\mathbb{C}[y],\,\deg P\geq 2
\end{align*}
can be viewed as a birational map of $\mathbb{P}^2(\mathbb{C})$ and $\lambda(f)=\deg P>1.$
\end{itemize}
\end{egs}

\noindent Let $f$ be a bimeromorphic map on a K\"{a}hler surface $X.$ To relate $\lambda(f)$ to the spectral radius of $f^*$ we need the equali\-ty~$(f^*)^n=(f^n)^*$ for all $n.$ When it occurs we say that $f$ is {\it analytically stable} (\cite{FS, S}). An other characterization can be found in \cite[Th. \!1.14]{DiFa}: the map $f$ is analytically stable if and only if there is no cur\-ve~$\mathcal{C}$ in $X$ such that $f^k(\mathcal{C}) \subset\mathrm{Ind}\, f$ for some integer $k\geq 0.$ Up to a birational change of coordinates, one can always arrange for a bimeromorphic map of a Kähler surface to be analytically stable (\emph{see} \cite[Th. \!0.1]{DiFa}). For instance, if $f$ is an automorphism, then $f$ is analytically stable and $\lambda (f)$ is the spectral radius of $f^*.$ Besides, since $f^*$ is defined over $\mathbb{Z},$ $\lambda (f)$ is also the spectral radius of $f_*.$

\par \medskip

\noindent Let us recall some properties about blowups of the complex projective plane. Let $p_1,$ $\ldots,$ $p_n$ be $n$ (possibly infinitely near) points in $\mathbb{P}^2(\mathbb{C})$ and $\mathrm{Bl}_{p_1,\ldots,p_n} \mathbb{P}^2$ denote the complex manifold obtained by blowing up $\mathbb{P}^2(\mathbb{C})$ at $p_1,$ $\ldots,$ $p_n.$

\begin{itemize}

\item We can identify $\mathrm{Pic}(\mathrm{Bl}_{p_1,\ldots,p_n}\mathbb{P}^2)$ and $\mathrm{H}^2(\mathrm{Bl}_{p_1,\ldots,p_n}\mathbb{P}^2, \mathbb{Z})$ so we won't make any difference in using them.

\item If $\pi\colon\mathrm{Bl}_{p_1,\ldots,p_n} \mathbb{P}^2\to \mathbb{P}^2(\mathbb{C})$ is the sequel of blowups of the $n$ points $p_1,$ $\ldots,$ $p_n,$ $\mathrm{H}$ the class of a generic line and $\mathrm{E}_j=\pi^{-1}(p_j)$ the exceptional fibers, then $\{\mathrm{H},\,\mathrm{E}_1,\,\ldots,\,\mathrm{E}_n\}$ is a basis of the free $\mathbb{Z}$-module
$\mathrm{Pic}(\mathrm{Bl}_{p_1,\ldots,p_n} \mathbb{P}^2).$

\item Assume that $n\leq 9$ and that $f\colon\mathrm{Bl}_{p_1,\ldots,p_n} \mathbb{P}^2\to\mathrm{Bl}_{p_1,\ldots,p_n}\mathbb{P}^2$ is an automorphism. Then the topological entropy of~$f$ vanishes. If $n\leq 8$ then there exists an integer $k$ such that $f^k$ descends to a linear map of $\mathbb{P}^2(\mathbb{C})$ (\emph{see} \cite[Prop. 2.2]{Di}).
\end{itemize}

\noindent In the sequel, $\mathbb{P}^2$ will denote the complex projective plane.

\subsection{Desingularization of birational maps}\label{desing}

\noindent In this section, we recall well-known results about birational maps between algebraic surfaces. We refer the reader to \cite[IV \S 3.4]{Sh} for more details.

\begin{itemize}
\item Every regular map $f \colon X \rightarrow Y$ between smooth projective surfaces which is birational can be written as $\varphi \circ \pi_N \circ \pi_{N-1} \circ \ldots \circ \pi_1,$ where the $\pi_i$'s are blowups and $\varphi$ is an isomorphism. Besides, the centers of the blowups are uniquely determined by $f.$

\item If $f \colon X \dashrightarrow Y$ is a rational map between smooth projective surfaces, there exist a canonical rational surface $\widetilde{X}$ obtained from $X$ by a finite sequence of blowups such that, if $\pi$ is the composition of the blowups, $f \circ \pi$ is a regular map. Any other rational surface satisfying the same property is obtained by blowing up $\widetilde{X}$ finitely many times. The surface $\widetilde{X}$ is called the minimal desingularization of $f.$

\item If $f \colon X \dashrightarrow Y$ is a birational map between smooth projective surfaces and if $\widetilde{X},$ $\widetilde{Y}$ are the minimal desingularizations of $f$ and $f^{-1},$ then $f$ induces an isomorphism between $\widetilde{X}$ and $\widetilde{Y}.$
\end{itemize}

\noindent Thus, for any rational map $f \colon X \dashrightarrow Y$  between smooth projective surfaces, there exist two canonical sets of (possibly infinitely) near points $\widehat{\xi}_1$ and $\widehat{\xi}_2$ in $X$ and $Y$ such that $f$ induces a canonical isomorphism between $\mathrm{Bl}_{\widehat{\xi}_1}X$ and $\mathrm{Bl}_{\widehat{\xi}_2}Y.$ We say that $\widehat{\xi}_1$ and $\widehat{\xi}_2$ correspond to the minimal desingularization of $f.$ In the sequel, we deal with \textit{ordered} sets of possibly infinitely near points of $\mathbb{P}^2,$ so that we usually order $\widehat{\xi}_1$ and $\widehat{\xi}_2.$ In that case, $\widehat{\xi}_1$ and $\widehat{\xi}_2$ are canonical up to a reordering.

\bigskip

\noindent One of the main properties of the minimal desingularization of a rational map is:

\begin{lem}\label{minimal}
{\sl Let $f \colon X \dashrightarrow Y$ is a rational map between smooth projective surfaces, $\widehat{\xi}_1$ and $\widehat{\xi}_2$ be ordered sets of (possibly infinitely) near points in $X$ and $Y$ corresponding to the minimal desingularization of $f$ and $G$ be the subgroup of $\mathrm{Aut}(X) \times \mathrm{Aut}(Y)$ consisting of couples $(A,B)$ such that $Af=fB.$ If $(A,B)$ is in the connected component of the identity of $G,$ then $A \, \widehat{\xi}_2=\widehat{\xi}_2$
and $B \, \widehat{\xi}_1=\widehat{\xi}_1.$}
\end{lem}

\begin{proof}
We argue by induction on the number of blowups in the minimal desingularization of $f.$ Let $(A,B)$ be in~$G$ and~$(A_t, B_t)_{\, 0 \leq t \leq 1}$ be a continuous path in $G$ between $(\mathrm{id},\mathrm{id})$ and $(A,B).$
For all $t$ in $[0,1]$ the matrix $B_t$ induces a permutation of the points of the locus of indeterminacy of $f.$ Since $B_0=\mathrm{id},$ each point of this locus must be fixed. Reasoning with $f^{-1}$ instead of $f,$ we obtain that $A$ fixes the points of indeterminacy of $f^{-1}.$ Let $p$ (resp. $q$) be one point of indeterminacy of $f$ (resp. $f^{-1}$). If $X'=\mathrm{Bl}_p X$ and $Y'=\mathrm{Bl}_q Y,$ $f$ induces a rational map $f' \colon X' \dashrightarrow Y'.$ Besides, if~$\widehat{\xi}'_{1}=~\widehat{\xi}_{1} \setminus \{p\}$ and $\widehat{\xi}'_{2}=\widehat{\xi}_{2} \setminus \{q\},$ then $\widehat{\xi}'_{1}$
and $\widehat{\xi}'_{2}$ correspond to the minimal desingularization of $f'.$ The automorphisms~$A$ and $B$ of $\mathbb{P}^2$ can be lifted to automorphisms $A'$ and $B'$ of $X'$ and $Y'$ respectively; and by induction, $A' \, \widehat{\xi}'_2=\widehat{\xi}'_2$
and $B' \, \widehat{\xi}'_1=\widehat{\xi}'_1.$ This yields the result.
\end{proof}

\noindent In \S \ref{casgal}, \S \ref{eg12} and \S \ref{eg2}, we will compute in concrete examples ordered sets of (possibly infinitely near) points of the complex projective plane corresponding to minimal desingularizations of Cremona transformations.

\subsection{Generic number of parameters of a family of Cremona transformations}\label{jabuse}

\noindent For every positive integer $d,$ let $\mathrm{Bir}_d(\mathbb{P}^2)$ be the set of birational maps of the complex projective plane given by a triplet of homogeneous polynomials of degree $d$ without common factors. Then $\mathrm{Bir}_d(\mathbb{P}^2)$ is a Zariski open subset of $\mathbb{P}^{\frac{3(d+1)(d+2)}{2}}.$ We will use two actions of the algebraic group $\mathrm{PGL}(3; \mathbb{C})$ on $\mathrm{Bir}_d(\mathbb{P}^2),$ namely:

\begin{itemize}
\item The \textit{left action}, given for $M$ in $\mathrm{PGL}(3; \mathbb{C})$ and $f$ in $\mathrm{Bir}_d(\mathbb{P}^2)$ by $M.f=Mf.$

\item The \textit{adjoint action}, given for $M$ in $\mathrm{PGL}(3; \mathbb{C})$ and $f$ in $\mathrm{Bir}_d(\mathbb{P}^2)$ by $M.f=MfM^{-1}.$
\end{itemize}

\noindent The associated orbits will be called \textit{left orbits} and \textit{adjoint orbits} in order to distinguish them. In this section, we will be mainly concerned with the adjoint action. From the point of view of holomorphic dynamical systems, two Cremona transformations belonging to the same adjoint orbit are essentially similar, since they are conjugate by a biholomorphism.
\par\medskip

\noindent By definition, a holomorphic family of birational maps of degree $d$ will be an irreducible analytic subset of~$\mathrm{Bir}_d(\mathbb{P}^2).$  To associate with a holomorphic family of Cremona transformations a generic number of parameters, we need a general result concerning holomorphic group actions:

\begin{lem} \label{stupide}
{\sl Let $G$ be a complex Lie group acting holomorphically on a complex manifold $X$ and $Y$ be an irreducible analytic subset of $X.$ Then for any generic element in $Y,$ the intersection of the orbit $O_y$ of $y$ with $Y$ is smooth in a neighborhood of $y,$ and its dimension is independent of $y.$}
\end{lem}

\begin{proof}
Let $\mathfrak{g}$ be the Lie algebra of $G.$ Each element $Z$ of $\mathfrak{g}$ induces
a holomorphic vector field $X_Z$ on $X$ corresponding to the infinitesimal action of $G$ on $X$ in the direction $Z.$ We fix a basis $Z_1, \ldots, Z_k$ of $\mathfrak{g}.$ Then, for any $x$ in $X,$ $\mathrm{Vect}(X_{Z_1}(x), \ldots, X_{Z_k}(x))$ is equal to $T_x O_x.$ This proves that for $y$ generic in $Y,$ the dimension of $T_y O_y \cap T_y Y$ is independent of $Y$; we call it $m.$ After removing a proper analytic subset out of $Y$ if necessary, we can assume that this property holds for all $y$ in $Y.$ Let $y$ be any point in $Y.$ Then, for any $\tilde{y}$ in $O_y \cap Y,$ $T_{\tilde{y}}(O_y \cap Y)=T_{\tilde{y}} O_y \cap T_{\tilde{y}} Y=T_{\tilde{y}} O_{\tilde{y}} \cap T_{\tilde{y}} Y,$ so that the Zariski tangent spaces of the analytic set $O_{y} \cap Y$ all have the same dimension $m.$ This implies that $O_{y} \cap Y$ is smooth of dimension $m.$
\end{proof}

\noindent We apply this lemma for $X=\mathrm{Bir}_d(\mathbb{P}^2)$ and $G=\mathrm{PGL}(3; \mathbb{C}).$ This justify the following definition:

\begin{defi}
If $d$ is a positive integer and $Y$ is a holomorphic family of birational maps of degree $d,$ we define the \textit{generic number of parameters of Y}, denoted by $\mathfrak{m}(Y),$ by $\mathfrak{m}(Y)= \mathrm{dim}\, Y-\mathrm{dim} (O_f \cap Y)^{irr},$ where $f$ is a generic element in $Y,$ $O_f$ is the adjoint orbit of $f$ in $\mathrm{Bir}_d(\mathbb{P}^2)$ and $(O_f \cap Y)^{irr}$ is the irreducible component of $f$ in $O_f \cap Y.$
\begin{itemize}
\item If $\mathfrak{m}(Y)=0,$ we say that $Y$ is \textit{holomorphically trivial}.
\item If $\mathfrak{m}(Y)=\mathrm{dim}\,Y,$ we say that $Y$ is \textit{generically effective}.
\end{itemize}
\end{defi}

\noindent If a holomorphic family $Y$ of Cremona transformations is holomorphically trivial, then for any generic point $y$ in $Y$, $O_y \cap Y$ is an open neighborhood of $y$ in $Y$ for the usual topology. We can be even more precise: for any point $f$ in $\mathrm{Bir}_d (\mathbb{P}^2),$ the adjoint orbit $O_f$ of $f$ is Zariski-open in its Zariski closure, so that $O_f \cap Y$ is an analytic subset of $Y.$ If $y$ is generic in $Y,$ then $\dim (O_y \cap Y)^{irr}=\dim\, Y.$ This implies that $(O_y \cap Y)^{irr}$ is Zariski open in $Y,$ so that all generic points of $Y$ lie in the same adjoint orbit. This means that the parameters of $Y$ are "fake" parameters.

 \par \smallskip

\noindent On the other hand, if a holomorphic family $Y$ of birational maps is generically effective, then for any generic element $f$ in $Y,$ there exists a neighborhood $U$ of $f$ in $Y$ such that $O_f \cap U=\{f\}.$ Besides, $U$ can be chosen open for the Zariski topology.

\par\medskip

\noindent In concrete examples, the generic number of parameters of a family of Cremona transformations can be computed easily using the following proposition:

\begin{pro}\label{critere0}
{\sl Let $Y$ be a holomorphic family of birational maps of dimension $n.$ Then $\mathfrak{m}(Y)$ is the smallest integer $k$ such that for a generic transformation $f$ in $Y,$ there exist a neighborhood $\Omega$ of $0$ in $\mathbb{C}^{n-k}$ and two holomorphic maps $\gamma \colon \Omega \rightarrow Y$ and $M \colon \Omega \rightarrow \mathrm{PGL}(3; \mathbb{C})$ such that:
 \begin{itemize}
\item $\gamma_{*}(\, 0)$ is injective,
\item $\gamma\, (\, 0)=f$ and $M(0\!)=\mathrm{Id},$
\item for all $t$ in $\Omega,$ $f_{\gamma \,(t \!)}=M(\, t)\, f\, M(\, t)^{-1}\,.$
\end{itemize}}
\end{pro}

\begin{proof}
\noindent Let $\gamma$ and $M$ satisfying the hypotheses of the proposition. Then the image of $\gamma$ near $f$ is a submanifold of $Y$ of dimension $n-k,$ which is included in the adjoint orbit of $f.$ This proves that $k$ is greater than or equal to $\mathfrak{m}(Y).$ To get the inequality in the opposite direction, we choose a generic transformation $f$ in $Y.$ By Lemma \ref{stupide}, the intersection $Z$ of the adjoint orbit of $f$ with $Y$ is smooth near $f.$ Since the orbit map from $\mathrm{PGL}(3; \mathbb{C})$ to $O_{f}$ is a holomorphic submersion, we can find locally around $f$ a holomorphic section $\tau \colon O_{f} \rightarrow \mathrm{PGL}(3; \mathbb{C})$ such that $\tau(f)=\mathrm{id}.$ Then $\tau(Z)$ is a submanifold of~$\mathrm{PGL}(3; \mathbb{C})$ of dimension $n-\mathfrak{m}(Y)$ passing through the identity. If $\Omega$ is a neighborhood of $0$ in $\mathbb{C}^{n-\mathfrak{m}(Y)}$ and $\gamma \, \colon \Omega \rightarrow Z$ is a local parametrization of $Z$ such that $\gamma(0)=f,$ we define $M \colon \Omega \rightarrow \tau(Z)$ by $M(t)=\tau (\gamma(t\!)).$ We obtain that for all $t$ in~$\Omega,$ $f_{\gamma \,(t \!)}=M(\, t)\, f\,  M(\, t)^{-1},$ this implies that $k$ is smaller than or equal to $\mathfrak{m}(Y).$
\end{proof}

\par\medskip

\begin{eg}\label{winner}
For each couple of integers $(n,k)$ with $n \geq 3$ and $k \geq 2,$ let us compute the number of parameters of the family $(f_a)$ of birational maps given in (\ref{bedfordkk}). We use the notations introduced in Theorem \ref{thmbk}.

\par\medskip
\noindent For a generic point $a$ in $\mathbb{C}^{\frac{n-3}{2}},$ let $\gamma\,  \colon \Omega \rightarrow S$ and $M \colon \Omega \rightarrow \mathrm{PGL}(3; \mathbb{C})$ be two holomorphic maps satisfying the assumptions of Proposition \ref{critere0}. Then for all $t$ in $\Omega,$ $f_{\gamma \,(t \!)}=M(\, t)\, f_{a}\, M(\, t)^{-1}\,.$ If $P=(0:0:1),$ $Q=(0:1:0),$ $\Delta=\{x=0\}$ and~$\Delta'=\{z=0\},$ then for any $b$ in $\mathbb{C}^{\frac{n-3}{2}},$ $\mathrm{Ind}\,f_b=\{Q\},$ $\mathrm{Exc}\,f_b=\Delta',$ $f_b(\Delta')=P,$ and $\Delta$ is the only invariant line under~$f_b.$ This implies that for all $t$ in $\Omega,$ the automorphism $M(t\!)$ of $\mathbb{P}^2$ fixes $P$ and $Q,$ and leaves invariant the two lines~$\Delta$ and $\Delta'.$ Thus we can write $M(\, t)=
\left[\begin{array}{ccc}a_t & 0 & 0\\
b_t & c_t & 0 \\
0 & 0 & 1
\end{array}\right]$ where $a,$ $b$ and $c$ are holomorphic functions on $\Omega.$ If we put this expression in the equality $f_{\gamma \,(t \!)}=M(\, t)\, f_{a}\, M(\, t)^{-1},$ we obtain easily that for all $t$ in $\Omega,$ $a_t=c_t=1$ and~$b_t=0.$ This proves that the family $(f_a)$ is generically effective.
\end{eg}

\section{Birational maps whose exceptional locus is a line, I}\label{casgal}

\noindent For any integer $n$ greater than or equal to $3,$ let us consider the birational map $\Phi_n$ defined by $\Phi_n=(xz^{n-1}+y^n:yz^{n-1}:z^n).$ If $P=(1:0:0)$ and $\Delta=\{z=0\},$ then $\mathrm{Ind}\,\Phi_n=\{P\}$ and $\mathrm{Exc}\,\Phi_n=\Delta.$ Besides, $\Phi_n(\Delta)=P.$

\par\medskip

\noindent In this section we construct for every integer $n \geq 3$ two ordered sets of infinitely near points of the complex projective plane $\widehat{\xi}_1$ and~$\widehat{\xi}_2$ corresponding to the minimal desingularization of $\Phi_n.$
Then we give theoric conditions to produce automorphisms $\varphi$ of~$\mathbb{P}^2$ such that $\varphi \,\Phi_n$ is conjugate to an automorphism on a rational surface obtained from $\mathbb{P}^2$ by successive blowups.

\subsection{First step: description of the sequence of blowups}\label{cons}

\noindent We start by the description of two points $\widehat{\xi}_1$ and $\widehat{\xi}_2$ infinitely near $P$ of length $2n-1$ corresponding to the minimal desingularization of $\Phi_n.$

\medskip

\noindent \textit{Convention}: if $\mathcal{D}$ (resp. $\mathcal{D}_i$) is a curve on a surface $X,$ we will denote by $\mathcal{D}_1$ (resp. $\mathcal{D}_{i+1}$) the strict transform of  this curve in $X$ blown up at a point.

\medskip

\noindent Let us blow up $P$ in the domain and in the range; set $y=u_1,$ $z=u_1v_1$ then $\Delta_1=\{v_1=0\}$ and the exceptional divisor $\mathrm{E}$ is given by $\{u_1=0\}.$ We can also set $y=r_1s_1,$ $z=s_1;$ in these coordinates $\mathrm{E}= \{s_1=0\}.$ We get $$\Phi_n\colon(u_1,v_1)\to(u_1,u_1v_1)_{(y,z)}\to\big(v_1^{n-1}+u_1:u_1v_1^{n-1}:u_1v_1^n\big)=\left( \frac{u_1v_1^{n-1}}{v_1^{n-1}+u_1},\frac{u_1v_1^n}{v_1^{n-1}+u_1}\right)_{(y,z)}\to\left(\frac{u_1v_1^{n-1}}{v_1^{n-1}+u_1},v_1\right)_{(u_1,v_1)};$$ hence $\mathrm{E}$ is fixed, $\Delta_1$ is blown down to $P_1=(0,0)_{(u_1,v_1)}=\mathrm{E}\cap\Delta_1$ and $P_1$ is a point of indeterminacy. One also can see that $$\Phi_n\colon(r_1,s_1)\to(r_1s_1,s_1)_{(y,z)}\to\big(1+r_1^ns_1:r_1s_1:s_1\big)\to\left(\frac{r_1s_1} {1+r_1^ns_1},\frac{s_1}{1+r_1^ns_1}\right)_{(y,z)}\to\left(r_1,\frac{s_1}{1+r_1^ns_1}\right)_{(r_1,s_1)}.$$

\noindent Then we blow up $P_1$ in the domain and in the range. Set $u_1=u_2,$ $v_1=u_2v_2$ so the exceptional divisor $\mathrm{F}$ is given by $\{u_2=0\}$ and~$\Delta_2$ by $\{v_2=0\}.$ There is an other system of coordinates $(r_2,s_2)$ with $u_1=r_2s_2,$ $v_1=s_2;$ in this system, $\mathrm{F}=\{s_2 =0\}$ and~$\mathrm{E}_1=\{r_2=0\}.$ On the one hand \begin{align*}
&\Phi_n\colon(u_2,v_2)\to(u_2,u_2v_2)_{(u_1,v_1)}\to\big(u_2^{n-2}v_2^{n-1}+1:u_2^{n-1}v_2^{n-1}:u_2^nv_2^n\big)=\left(\frac{u_2^{n-1}v_2^{n-1}}{u_2^{n-2}v_2^{n-1}+1},\frac{u_2^nv_2^n}{u_2^{n-2}v_2^{n-1}+1}\right)_{(y,z)}\\
&\hspace{1cm}\to\left(\frac{u_2^{n-1}v_2^{n-1}}{u_2^{n-2}v_2^{n-1}+1},u_2v_2\right)_{(u_1,v_1)}\to\left(\frac{u_2^{n-2}v_2^{n-2}}{u_2^{n-2}v_2^{n-1}+1},u_2v_2\right)_{(r_2,s_2)};
\end{align*}
\noindent on the other hand
\begin{align*}
&\Phi_n\colon(r_2,s_2)\to(r_2s_2,s_2)_{(u_1, v_1)}\to\big(s_2^{n-2}+r_2:r_2s_2^{n-1}:r_2s_2^n\big)\to\left(\frac{r_2s_2^{n-1}}{s_2^{n-2}+r_2},\frac{r_2s_2^n}{s_2^{n-2}+r_2}\right)_{(y,z)}\\
&\hspace{1cm}\to\left(\frac{r_2s_2^{n-1}}{s_2^{n-2}+r_2},s_2\right)_{(u_1,v_1)}\to\left(\frac{r_2s_2^{n-2}}{s_2^{n-2}+r_2},s_2\right)_{(r_2,s_2)}.
\end{align*}

\noindent So $A_1=\mathrm{E}_1\cap\mathrm{F}=(0,0)_{(r_2,s_2)}$ is a point of indeterminacy and $\mathrm{F}$ is blown down to $A_1.$

\par \medskip

\noindent Moreover $$\Phi_n\colon(y,z)\to(z^{n-1}+y^n:yz^{n-1}:z^n)\to\left(\frac{yz^{n-1}}{y^n+z^{n-1}},\frac{z^n}{y^n+z^{n-1}}\right)_{(y,z)}\to\left(\frac{yz^{n-1}}{y^n+z^{n-1}},\frac{z}{y}\right)_{(u_1,v_1)}\to\left(\frac{y^2z^{n-2}}{y^n+z^{n-1}},\frac{z}{y}\right)_{(r_2,s_2)};$$

\noindent hence $\Delta_2$ is blown down to $A_1.$

\medskip

\noindent The $(n-3)$ next steps are of the same type, so we will write it one time with some indice $k,$ $1\leq k\leq n-3.$

\noindent We will blow up $A_k=\mathrm{E}_k\cap\mathrm{G}^{k-1}=v(0,0)_{(r_{k+1},s_{k+1})}$ in the domain and in the range. Set
\begin{align*}
&r_{k+1}=u_{k+2}, &&s_{k+1}=u_{k+2}v_{k+2} &&\& &&r_{k+1}=r_{k+2}s_{k+2}, &&s_{k+1}=s_{k+2}.
\end{align*}

\noindent Let us remark that $(u_{k+2},v_{k+2})$ (resp. $(r_{k+2},s_{k+2})$) is a system of coordinates in which the exceptional divisor $\mathrm{G}^k$ is given by $\mathrm{G}^k=\{u_{k+2}=0\}$ and $\mathrm{G}_1^{k-1}=\{v_{k+2}=0\}$ (resp. $\mathrm{G}^k=\{s_{k+2}=0\}$ and  $\mathrm{E}_{k+1}=\{r_{k+2}=0\}$). We have
\begin{align*}
&\Phi_n\colon(u_{k+2},v_{k+2})\to(u_{k+2},u_{k+2}v_{k+2})_{(r_{k+1},s_{k+1})}\to\big(1+u_{k+2}^{n-k-2}v_{k+2}^{n-k-1}:u_{k+2}^{n-1}v_{k+2}^{n-1}:u_{k+2}^nv_{k+2}^n\big)\\
&\hspace{1cm}\to\left(\frac{u_{k+2}^{n-k-2}v_{k+2}^{n-k-2}}{1+u_{k+2}^{n-k-2}v_{k+2}^{n-k-1}},u_{k+2}v_{k+2}\right)_{(r_{k+2},s_{k+2})}
\end{align*}
\noindent and
\begin{align*}
&\Phi_n\colon(r_{k+2},s_{k+2})\to(r_{k+2}s_{k+2},s_{k+2})_{(r_{k+1},s_{k+1})} \to\Big(r_{k+2}+s_{k+2}^{n-k-2}:r_{k+2}s_{k+2}^{n-1}:r_{k+2}s_{k+2}^{n}\big)\\
&\hspace{1cm}\to\left(\frac{r_{k+2}s_{k+2}^{n-k-2}}{r_{k+2}+s_{k+2}^{n-k-2}},s_{k+2}\right)_{(r_{k+2},s_{k+2})}.
\end{align*}

\noindent So $A_{k+1}=\mathrm{G}^k\cap\mathrm{E}_{k+1}=(0,0)_{(r_{k+2}, s_{k+2})}$ is a point of indeterminacy and $\mathrm{G}^k,$ $\mathrm{G}_1^{k-1}$ are blown down to $A_{k+1}.$ Therefore $$\Phi_n\colon(y,z)\to\big(z^{n-1}+y^n:yz^{n-1}:z^n\big)=\left(\frac{yz^{n-1}}{z^{n-1}+y^n},\frac{z^n}{z^{n-1}+y^n}\right)_{(y,z)}\to\left(\frac{y^{k+2}z^{n-k-2}}{z^{n-1}+y^n},\frac{z}{y}\right)_{(r_{k+2},s_{k+2})}$$

\noindent so $\Delta_{k+2}$ is also blown down to $A_{k+1}.$

\noindent One can remark that $$\Phi_n\colon(u_2,v_2)\to\left(\frac{u_2^{n-2}v_2^{n-2}}{1+u_2^{n-2}v_2^{n-1}},u_2v_2\right)_{(r_2,s_2)}\to\left(\frac{u_2^{n-k-2}v_2^{n-k-2}}{1+u_2^{n-2}v_2^{n-1}},u_2v_2\right)_{(r_{k+2},s_{k+2})},$$

\noindent hence $\mathrm{F}_k$ is blown down to $A_{k+1}.$ Using a similar computation one can verify that $\mathrm{G}_j^{k-j}$ is blown down to $A_{k+1}.$

\medskip

\noindent Let us now blow up $A_{n-2}$ in the domain and in the range. Set $r_{n-1}=u_n,$ $s_{n-1}=u_nv_n,$ and $r_{n-1}=r_ns_n,$ $s_{n-1}=s_n.$ Let us remark that $(u_n,v_n)$ (resp. $(r_n,s_n)$) is a system of coordinates in which the exceptional divisor is given by $\mathrm{G}^{n-2}=\{u_n=0\}$ (resp. $\mathrm{G}^{n-2}=\{s_n=0\}$), and~$\mathrm{G}^{n-3}_1=\{v_n=0\}.$ We compute:
\begin{align*}
&\Phi_n\colon(u_n,v_n)\to(u_n,u_nv_n)_{(r_{n-1},s_{n-1})}\to\big(1+v_n:u_n^{n-1}v_n^{n-1}:u_n^nv_n^n\big)\to\left(\frac{1}{1+v_n},u_nv_n\right)_{(r_n,s_n)}
\end{align*}
\noindent and
\begin{align*}
&\Phi_n\colon(r_n,s_n)\to(r_ns_n,s_n)_{(r_{n-1},s_{n-1})} \to\big(1+r_n:r_ns_n^{n-1}:r_ns_n^n\big)\to\left(\frac{r_n}{1+r_n},s_n\right)_{(r_n,s_n)}.
\end{align*}

\noindent This implies that $\mathrm{G}^{n-2}$ is fixed, $\mathrm{G}^{n-3}_1$ is blown down to the point $S=(1,0)_{(r_n,s_n)}$ of $\mathrm{G}^{n-2}$ and the point $T=(-1,0)_{(r_n,s_n)}$ of $\mathrm{G}^{n-2}$ is a point of indeterminacy.

 \medskip

\noindent On the one hand $$\Phi_n\colon(y,z)\to\left(\frac{y^{n-1}z}{z^{n-1}+y^n},\frac{z}{y}\right)_{(r_{n-1},s_{n-1})}\to\left(\frac{y^n}{z^{n-1}+y^n},\frac{z}{y}\right)_{(r_n,s_n)}$$

\noindent so $\Delta_n$ is blown down to $S;$ on the other hand $$\Phi_n\colon(u_2,v_2)\to\left(\frac{u_2^{n-2}v_2^{n-2}}{1+u_2^{n-2}v_2^{n-1}},u_2v_2\right)_{(r_2,s_2)}\to\left(\frac{u_2v_2}{1+u_2^{n-2}v_2^{n-1}},u_2v_2\right)_{(r_{n-1},s_{n-1})}\to\left(\frac{1}{1+u_2^{n-2}v_2^{n-1}},u_2v_2\right)_{(r_n,s_n)},$$

\noindent hence $\mathrm{F}_{n-2}$ is blown down to $S.$

\medskip

\noindent Now we blow up $T$ in the domain and $S$ in the range

\begin{align*}
&\left\{\begin{array}{ll} r_n=u_{n+1}-1\\ s_n=u_{n+1}v_{n+1}\end{array}\right. && \begin{array}{ll} \mathrm{H}=\{u_{n+1}=0\}\\ \end{array} && \hspace{2cm} &&\left\{\begin{array}{ll} r_n=a_{n+1}+1\\ s_n=a_{n+1}b_{n+1}\end{array}\right. && \begin{array}{ll} \mathrm{K}=\{a_{n+1} =0\}\\ \end{array}
\end{align*}

\begin{align*}
&\left\{\begin{array}{ll} r_n=r_{n+1}s_{n+1}-1\\ s_n=s_{n+1}\end{array}\right. && \begin{array}{ll} \mathrm{H}=\{s_{n+1}=0\}\\ \end{array} && \hspace{1cm} &&\left\{\begin{array}{ll} r_n=c_{n+1} d_{n+1}+1\\ s_n=d_{n+1}\end{array}\right. && \begin{array}{ll} \mathrm{K}=\{d_{n+1} =0\}\\ \end{array}
\end{align*}

\noindent We obtain
\begin{align*}
&\Phi_n\colon(u_{n+1},v_{n+1})\to(u_{n+1}-1,u_{n+1}v_{n+1})_{(r_n,s_n)}\to\big(1:(u_{n+1}-1)u_{n+1}^{n-2}v_{n+1}^{n-1}:(u_{n+1}-1)u_{n+1}^{n-1}v_{n+1}^n\big)\\
&\hspace{1cm}=\big((u_{n+1}-1)u_{n+1}^{n-2}v_{n+1}^{n-1},(u_{n+1}-1)u_{n+1}^{n-1}v_{n+1}^n\big)_{(y,z)}\to\big((u_{n+1}-1)u_{n+1}^{n-2}v_{n+1}^{n-1},u_{n+1}v_{n+1}\big)_{(u_1,v_1)}\\
&\hspace{1cm}\to\big((u_{n+1}-1)u_{n+1}^{n-3}v_{n+1}^{n-2},u_{n+1}v_{n+1}\big)_{(r_2,s_2)}\to\ldots\to\big((u_{n+1}-1)v_{n+1},u_{n+1}v_{n+1}\big)_{(r_{n-1},s_{n-1})}
\end{align*}

\noindent and

\begin{align*}
&\Phi_n\colon(r_{n+1},s_{n+1})\to(r_{n+1}s_{n+1}-1,s_{n+1})_{(r_n,s_n)}\to\big(r_{n+1}:(r_{n+1}s_{n+1}-1) s_{n+1}^{n-2}:(r_{n+1}s_{n+1}-1)s_{n+1}^{n-1}\big)\\
&\hspace{1cm}=\left(\frac{(r_{n+1}s_{n+1}-1) s_{n+1}^{n-2}}{r_{n+1}},\frac{(r_{n+1}s_{n+1}-1)s_{n+1}^{n-1}} {r_{n+1}}\right)_{(y,z)}\to\left(\frac{(r_{n+1}s_{n+1}-1)s_{n+1}^{n-2}}{r_{n+1}},s_{n+1} \right)_{(u_1,v_1)}\\
&\hspace{1cm}\to\left(\frac{(r_{n+1}s_{n+1}-1)s_{n+1}^{n-3}}{r_{n+1}},s_{n+1} \right)_{(r_2,s_2)}\to\ldots\to\left(\frac{r_{n+1}s_{n+1}-1}{r_{n+1}},s_{n+1} \right)_{(r_{n-1},s_{n-1})}.
\end{align*}

\noindent Thus $\mathrm{H}$ is sent on $\mathrm{G}_2^{n-3}$ and $B_1=(0,0)_{(r_{n+1},s_{n+1})}$ is a point of indeterminacy. Moreover,

$$\Phi_n\colon(u_n,v_n)\to\left(\frac{1}{1+v_n},u_nv_n\right)_{(r_n,s_n)}\to\left( -\frac{v_n}{1+v_n},-u_n(1+v_n)\right)_{(a_{n+1},b_{n+1})},$$ $$\Phi_n\colon(y,z)\to\left(\frac{y^n}{z^{n-1} +y^n},\frac{z}{y}\right)_{(r_n,s_n)}\to\left(-\frac{yz^{n-2}}{z^{n-1}+y^n},\frac{z}{y}\right)_{(c_{n+1},d_{n+1})} $$ \noindent and $$\Phi_n\colon(u_2,v_2)\to\left(\frac{1}{1+u_2^{n-2}v_2^{n-1}}, u_2v_2\right)_{(r_n,s_n)}\to\left(-\frac{u_2^{n-3}v_2^{n-2}}{1+u_2^{n-2}v_2^{n-1}},u_2v_2\right)_{(c_{n+1},d_{n+1})} .$$ Therefore $\mathrm{G}_2^{n-3}$ is sent on $\mathrm{K}$ and $\Delta_{n+1},$ $\mathrm{F}_{n-1}$ are blown down to $C_1=(0,0)_{(c_{n+1},d_{n+1})}.$

\medskip

\noindent The $(n-3)$ following steps are the same, so we will write it one time with some indice $\ell,$ $1\leq\ell\leq n-3.$

\noindent We blow up $B_\ell=(0,0)_{(r_{n+\ell},s_{n+\ell})}$ in the domain and $C_\ell=(0,0)_{ (c_{n+\ell},d_{n+\ell})}$ in the range

\begin{align*}
&\left\{\begin{array}{ll} r_{n+\ell}=u_{n+\ell+1}\\ s_{n+\ell}=u_{n+\ell+1}v_{n+\ell+1}\end{array} \right. && \begin{array}{ll} \mathrm{L}^\ell=\{u_{n+\ell+1}=0\}\\ \end{array} && \hspace{1cm} &&\left\{\begin{array}{ll} c_{n+\ell}=a_{n+\ell+1}\\ d_{n+\ell}=a_{n+\ell+1}b_{n+\ell+1}\end{array} \right. && \begin{array}{ll} \mathrm{M}^\ell=\{a_{n+\ell+1} =0\}\\ \end{array}
\end{align*}

\begin{align*}
&\left\{\begin{array}{ll} r_{n+\ell}=r_{n+\ell+1}s_{n+\ell+1}\\ s_{n+\ell}=s_{n+\ell+1}\end{array} \right. && \begin{array}{ll} \mathrm{L}^\ell=\{s_{n+\ell+1}=0\}\\ \end{array} && \hspace{1cm} &&\left\{\begin{array}{ll} c_{n+\ell}=c_{n+\ell+1} d_{n+\ell+1}\\ d_{n+\ell}=d_{n+\ell+1} \end{array}\right. && \begin{array}{ll} \mathrm{M}^\ell=\{d_{n+\ell+1} =0\}\\ \end{array}
\end{align*}

\noindent On the one hand
\begin{align*}
&\Phi_n\colon(u_{n+\ell+1},v_{n+\ell+1})\to(u_{n+\ell+1},u_{n+\ell+1}v_{n+\ell+1})_{(r_{n+\ell},s_{n+\ell})}\\
&\hspace{1cm}\to \big(1:(u_{n+\ell+1}^{\ell+1}v_{n+\ell+1}^\ell-1) u_{n+\ell+1}^{n-\ell-2}v_{n+\ell+1}^{n-\ell-1}:(u_{n+\ell+1}^{\ell+1} v_{n+\ell+1}^\ell-1) u_{n+\ell+1}^{n-\ell-1}v_{n+\ell+1}^{n-\ell}\big)\\
&\hspace{1cm}\to\big((u_{n+\ell+1}^{\ell+1}v_{n+\ell+1}^\ell-1) v_{n+\ell+1},u_{n+\ell+1}v_{n+\ell+1}\big)_{(r_{n- 1-\ell},s_{n-1-\ell})}
\end{align*}

\noindent and on the other hand
\begin{align*}
&\Phi_n\colon(r_{n+\ell+1},s_{n+\ell+1})\to(r_{n+\ell+1}s_{n+\ell+1},s_{n+\ell+1})_{(r_{n+\ell},s_{n+\ell})}\\
&\hspace{1cm}\to\big(r_{n+\ell+1}:(r_{n+\ell+1}s_{n+\ell+1}^{\ell+1}-1) s_{n+\ell+1}^{n-\ell-2}:(r_{n+\ell+1}s_{n+\ell+1}^{\ell+1}-1) s_{n+\ell+1}^{n-\ell-1}\big)\\
&\hspace{1cm}\to\left(\frac{r_{n+\ell+1}s_{n+\ell+1}^{\ell+1}-1}{r_{n+\ell+1}},s_{n+\ell+1}\right)_{(r_{n- 1-\ell},s_{n-1-\ell})}.
\end{align*}

\noindent So $B_{\ell+1}=(0,0)_{(r_{n+\ell+1},s_{n+\ell+1})}$ is a point of indeterminacy, $\mathrm{L}^\ell$ is sent on $\mathrm{G}^{n-3-\ell}_{2\ell+2}$ if $1\leq \ell\leq n-4$ and $\mathrm{L}^{n-3}$ is sent on~$\mathrm{F}_{2n-4}.$ Remark that $B_{\ell+1}$ is on $\mathrm{L}^\ell$ but not on $\mathrm{L}^{\ell-1}_1.$ Besides one can verify that
\begin{align*}
&\Phi_n\colon(y,z)\to\left(-\frac{y^{\ell+1}z^{n-\ell-2}}{z^{n-1}+y^n}, \frac{z}{y}\right)_{(c_{n+\ell+1},d_{n+\ell+1})}&& \text{and that} &&\Phi_n\colon(u_2,v_2)\to\left(-\frac{u_2^{n-\ell -3}v_2^{n-\ell-2}}{1+u_2^{n-2}v_2^{n-1}},u_2v_2\right)_{(c_{n+\ell+1},d_{n+\ell+1})};
\end{align*}

\noindent thus $\Delta_{n+\ell+1}$ and $\mathrm{F}_{n+\ell-1}$ are blown down to $C_{\ell+1}=(0,0)_{( c_{n+\ell+1},d_{n+\ell+1})}$ for every $\ell<n-3.$ The situation is different for~$\ell=n-3:$ whereas $\Delta_{2n-2}$ is still blown down to $C_{n-2}=(0,0)_{(c_{2n-2},d_{2n-2})},$ the divisor $\mathrm{F}_{2n-4}$ is sent on $\mathrm{M}^{n-3}.$

\noindent One can also note that
\begin{align*} &(u_{k+2},v_{k+2})\to\left(\frac{u_{k+2}^{n-k-2}v_{k+2}^{n-k-2}}{1+u_{k+2}^{n-k-2}v_{k+2}^{n-k-1}},u_{k+2}v_{k+2}\right)_{(r_{k+2},s_{k+2})}\to\ldots\to\left(\frac{1}{1+u_{k+2}^{n-k-2}v_{k+2}^{n-k-1}},u_{k+2}v_{k+2}\right)_{(r_n,s_n)}\\
&\hspace{1cm}\to\left(-\frac{u_{k+2}^{n-k-3}v_{k+2}^{n-k-2}}{1+u_{k+2}^{n-k-2}v_{k+2}^{n-k-1}},u_{k+2}v_{k+2}\right)_{(c_{n+1},d_{n+1})}\to\ldots\to\left(-\frac{v_{k+2}}{1+u_{k+2}^{n-k-2}v_{k+2}^{n-k-1}},u_{k+2}v_{k+2}\right)_{(c_{2n-k-2},d_{2n-k-2})}
\end{align*}
\noindent so $\mathrm{G}^k_{2n-4-k}$ is sent on $\mathrm{M}^{n-k-3}_k$ for all $1\leq k\leq n-4.$

\medskip

\noindent Finally we blow up $B_{n-2}$ in the domain and $C_{n-2}$ in the range.

\begin{align*}
&\left\{\begin{array}{ll} r_{2n-2}=u_{2n-1}\\ s_{2n-2}=u_{2n-1}v_{2n-1}\end{array} \right. && \begin{array}{ll} \mathrm{L}^{n-2}=\{u_{2n-1}=0\}\\ \end{array} && \hspace{1cm} &&\left\{\begin{array}{ll} c_{2n-2}=a_{2n-1}\\ d_{2n-2}=a_{2n-1}b_{2n-1}\end{array} \right. && \begin{array}{ll} \mathrm{M}^{n-2}=\{a_{2n-1} =0\}\\ \end{array}
\end{align*}

\begin{align*}
&\left\{\begin{array}{ll} r_{2n-2}=r_{2n-1}s_{2n-1}\\ s_{2n-2}=s_{2n-1}\end{array} \right. && \begin{array}{ll} \mathrm{L}^{n-2}=\{s_{2n-1}=0\}\\ \end{array} && \hspace{1cm} &&\left\{\begin{array}{ll} c_{2n-2}=c_{2n-1} d_{2n-1}\\ d_{2n-2}=d_{2n-1} \end{array}\right. && \begin{array}{ll} \mathrm{M}^{n-2}=\{d_{n-2}=0\}\\ \end{array}
\end{align*}

\noindent This yields
\begin{align*}
&\Phi_n\colon(u_{2n-1},v_{2n-1})\to(u_{2n-1},u_{2n-1}v_{2n-1})_{(r_{2n-2},s_{2n-2})}\to \big(1:(u_{2n-1}^{n-1}v_{2n-1}^{n-2}-1) v_{2n-1}:(u_{2n-1}^{n-1} v_{2n-1}^{n-2}-1) u_{2n-1}v_{2n-1}^2\big)
\end{align*}

\noindent and
\begin{align*}
&\Phi_n\colon(r_{2n-1},s_{2n-1})\to(r_{2n-1}s_{2n-1},s_{2n-1})_{(r_{2n-2},s_{2n-2})}\to\big(r_{2n-1}:r_{2n-1}s_{2n-1}^{n-1}-1:(r_{2n-1}s_{2n-1}^{n-1}-1) s_{2n-1}\big).
\end{align*}

\noindent Thus there is no point of indeterminacy and $\mathrm{L}^{n-2}$ is sent on $\Delta_{2n-1}.$

Furthermore,
$$(y,z)\to\left(-\frac{y^{n-1}}{z^{n-1}+y^n}, \frac{z}{y}\right)_{(c_{2n-1},d_{2n-1})}$$

\noindent so $\Delta_{2n-1}$ is sent on $\mathrm{M}^{n-2}.$

\medskip

\noindent All these computations yield the following result:

\begin{pro} \label{degqcq}
{\sl Let $\widehat{\xi}_1$ (resp. $\widehat{\xi}_2$) denote the point infinitely near $P$ obtained by blowing up $P,$ $P_1,$ $A_1,$ $\ldots,$ $A_{n-2},$ $T,$ $B_1,$ $\ldots,$ $B_{n-3}$ and $B_{n-2}$ (resp. $P,$ $P_1,$ $A_1,$ $\ldots,$ $A_{n-2},$ $S,$ $C_1,$ $\ldots,$ $C_{n-3}$ and $C_{n-2}$).
Then $\widehat{\xi}_1$  and $\widehat{\xi}_2$ correspond to the minimal desingularization of $\Phi_n,$ so that $\Phi_n$ induces an isomorphism between $\mathrm{Bl}_{\widehat{\xi}_1}\mathbb{P}^2$ and $\mathrm{Bl}_{\widehat{\xi}_2}\mathbb{P}^2.$ The different components are swapped as follows:
\begin{align*}
&&\Delta\to\mathrm{M}^{n-2}, &&\mathrm{E}\to\mathrm{E}, && \mathrm{F}\to\mathrm{M}^{n-3}, &&\mathrm{G}^{n-3}\to\mathrm{K}, && \mathrm{G}^{n-2}\to\mathrm{G}^{n-2}, &&\mathrm{H}\to\mathrm{G}^{n-3},   && \mathrm{L}^{n-3}\to \mathrm{F},  &&\mathrm{L}^{n-2}\to\Delta,
\end{align*}
\begin{align*}
&\mathrm{G}^k\to\mathrm{M}^{n-k-3}\,\, \text{ for } 1\leq k\leq n-4,
&&\mathrm{L}^\ell\to \mathrm{G}^{n-3-\ell}\,\, \text{ for } 1\leq \ell\leq n-4.
\end{align*}}
\end{pro}

\noindent The sequence of blowups corresponding to the infinitely near points $\widehat{\xi}_1$ and $\widehat{\xi}_2$ already appeared in \cite{BK3}.

\subsection{Second step: gluing conditions}

\noindent The gluing conditions reduce to the following problem: if $u$ is a germ of biholomorphism in a neighborhood of $P,$ find the conditions on $u$ in order that $u(\widehat{\xi}_2)=\widehat{\xi}_1.$

\medskip

\noindent Consider a neighborhood of $(0,0)$ in $\mathbb{C}^2$ with the coordinates $\eta_1,$ $\mu_1.$ For every integer $d\geq 1,$ we introduce an infinitely near point $\widehat{\Omega}_d$ of length $d$ centered at $(0,0)$ by blowing up successively $\omega_1,$ $\ldots,$ $\omega_d$, where
$\omega_i=(0,0)_{(\eta_i,\mu_i)}$ and the coordina\-tes~$(\eta_i,\mu_i)$ are given by the formulae $\eta_i=\eta_{i+1} \mu_{i+1},$ $\mu_i=\mu_{i+1}.$

\noindent Let $g(\eta_1,\mu_1)=\left(\displaystyle\sum_{(i,j)\in\mathbb{N}^2}\alpha_{i,j}\eta_1^i\mu_1^j,\displaystyle\sum_{(i,j)\in\mathbb{N}^2}\beta_{i,j}\eta_1^i\mu_1^j\right)_{(\eta_1,\mu_1)}$ be a germ of biholomorphism at $(0,0)_{(\eta_1,\mu_1)}.$ If $d$ is a positive integer, we define the subset $I_d$ of $\mathbb{N}^2$  by $I_d=\{(0,0),\, (0,1), \, \ldots, \, (0,d-1) \}.$

\begin{lem}\label{glucond}
{\sl If $d$ is in $\mathbb{N}^*,$ $g$ can be lifted to a biholomorphism $\widetilde{g}$ in a neighborhood of the exceptional components in $\mathrm{Bl}_{\widehat{\Omega}_d}\mathbb{C}^2$ if and only if $\alpha_{0,0}=\beta_{0,0}=0$ and $\alpha_{0,1}=\ldots=\alpha_{0,d-1}=0.$
If these conditions are satisfied, $\beta_{0,1}\not=0$ and $\widetilde{g}$ is given in the coordinates $(\eta_{d+1},\mu_{d+1})$ by the formula
\begin{align*}
&\widetilde{g}(\eta_{d+1},\mu_{d+1})=\left(\frac{\displaystyle\sum_{(i,j)\not\in I_d}\alpha_{i,j}\eta_{d+1}^i\mu_{d+1}^{d(i-1)+j}}{\displaystyle\sum_{(i,j)\not\in I_1}\beta_{i,j}\eta_{d+1}^i\mu_{d+1}^{di+j-1}},\displaystyle\sum_{(i,j)\not\in I_1}\beta_{i,j}\eta_{d+1}^i\mu_{d+1}^{di+j}\right)_{(\eta_{d+1},\mu_{d+1})}, && \vert\mu_{d+1}\vert<\varepsilon,\,\vert\eta_{d+1}\mu_{d+1}\vert< \varepsilon
\end{align*}}
for $\epsilon$ sufficiently small.
\end{lem}

\begin{proof}
This is straightforward by induction on $d.$
\end{proof}

\noindent Fix $n\geq 3,$ then $\mathrm{Bl}_{\widehat{\xi}_1}\mathbb{C}^2$ can be obtained as follows:
\begin{itemize}
\item blow up $P$ ;

\item blow up $\widehat{\Omega}_{n-1}$ centered at $P_1$ ({\it i.e.} $\eta_1=u_1,$ $\mu_1=v_1$) ;

\item blow up $\widehat{\Omega}_{n-1}$ centered at $T$ ({\it i.e.} $\eta_1=r_n+1,$ $\mu_1=s_n$).
\end{itemize}

\noindent The same holds with $\widehat{\xi}_2,$ the point $T$ being replaced by $S.$

\begin{pro}\label{gluglu}
{\sl Let $u(y,z)=\left(\displaystyle\sum_{(i,j)\in\mathbb{N}^2} m_{i,j}y^iz^j, \displaystyle\sum_{ (i,j)\in\mathbb{N}^2} n_{i,j}y^iz^j\right)$ be a germ of biholomorphism at $P.$

\begin{itemize}
\item If $n=3,$ then $u$ can be lifted to a germ of biholomorphism between $\mathrm{Bl}_{\widehat{\xi}_2}\mathbb{P}^2$ and  $\mathrm{Bl}_{\widehat{\xi}_1}\mathbb{P}^2$ if and only if
\begin{itemize}
\item $m_{0,0}=n_{0,0}=0;$

\item $n_{1,0}=0;$

\item $m_{1,0}^3+n_{0,1}^2=0;$

\item $n_{2,0}=\frac{3m_{0,1}n_{0,1}}{2m_{1,0}}.$
\end{itemize}

\item If $n\geq 4,$ then $u$ can be lifted to a germ of biholomorphism between $\mathrm{Bl}_{\widehat{\xi}_2}\mathbb{P}^2$ and  $\mathrm{Bl}_{\widehat{\xi}_1}\mathbb{P}^2$ if and only if
\begin{itemize}
\item $m_{0,0}=n_{0,0}=0;$

\item $n_{1,0}=0;$

\item $m_{1,0}^n+n_{0,1}^{n-1}=0;$

\item $m_{0,1}=n_{2,0}=0.$
\end{itemize}
\end{itemize}}
\end{pro}

\begin{proof}
The first condition is $u(P)=P,$ {\it i.e.} $m_{0,0}=n_{0,0}=0.$ The associated lift $\widetilde{u}_1$ is given by $$\widetilde{u}_1(u_1,v_1) =\left(\displaystyle\sum_{\begin{footnotesize}(i,j)\not\in I_1\end{footnotesize}}m_{i,j}u_1^{i+j}v_1^j, \frac{\displaystyle\sum_{\begin{footnotesize}(i,j)\not\in I_1\end{footnotesize}}n_{i,j}u_1^{i+j-1}v_1^j}{\displaystyle\sum_{\begin{footnotesize}(i,j)\not\in I_1\end{footnotesize}}m_{i,j}u_1^{i+j-1} v_1^j}\right)_{(u_1,v_1)}.$$ We must now verify the gluing conditions of Lemma \ref{glucond} for $g=\widetilde{u}_1$ with $d=n-1.$ This implies only the condition $n_{1,0}=0,$ since $\alpha_{0,j}(\widetilde{u}_1)=0$ for $j\geq 0.$ After blowing up $\widehat{\Omega}_{n-1}$ we get $$\widetilde{u}_n(r_n,s_n)=\left(\frac{\left(\displaystyle\sum_{\begin{footnotesize}(i,j)\not\in I_1\end{footnotesize}}m_{i,j} r_n^{i+j}s_n^{(n-1)(i+j-1)+j}\right)^n}{\left(\displaystyle\sum_{\begin{footnotesize}\substack{(i,j)\not\in I_1\\(i,j)\not=(1,0)\phantom{-}}\end{footnotesize}}n_{i,j}r_n^{i+j-1}s_n^{(n-1)(i+j-1)+j-1}\right)^{n-1}},\frac{\displaystyle\sum_{\begin{footnotesize}\substack{(i,j)\not\in I_1\\(i,j)\not=(1,0)\phantom{-}}\end{footnotesize}}n_{i,j}r_n^{i+j-1}s_n^{(n-1)(i+j-1)+j}}{\displaystyle\sum_{\begin{footnotesize}(i,j)\not\in I_1\end{footnotesize}}m_{i,j}r_n^{i+j-1}s_n^{(n-1)(i+j-1)+j}}\right)_{(r_n,s_n)}.$$
The condition $\widetilde{u}_n(S)=T$ is equivalent to $\frac{m_{1,0}^n} {n_{0,1}^{n-1}}=-1.$ In the coordinates $(\eta_1,\mu_1)$ centered at $S$ and $T,$ $$\widetilde{u}_n(\eta_1,\mu_1)=(\Gamma_1(\eta_1, \mu_1),\Gamma_2(\eta_1,\mu_1))_{(\eta_1,\mu_1)}$$ where $$\Gamma_1(\eta_1,\mu_1)=\frac{\left(
\displaystyle\sum_{\begin{footnotesize}(i,j)\not\in I_1\end{footnotesize}}m_{i,j} (1+\eta_1)^{i+j}\mu_1^{(n-1)(i+j-1)+j}\right)^{n-1}}{\displaystyle\left(\sum_{\begin{footnotesize}(i,j)\not\in I_1\end{footnotesize}}n_{i,j}(1+\eta_1)^{i+j-1}\mu_1^{(n-1)(i+j-1)+j-1}\right)^{n-1}}+1.$$ Thus $$\Gamma_1(0,\mu_1)=\frac{\big(m_{1,0}+m_{0,1}\mu_1+o(\mu_1^{n-2})\big)^n}{\big(n_{0,1}+n_{2,0}
\mu_1^{n-2}+o(\mu_1^{n-2})\big)^{n-1}}+1.$$ The gluing conditions of higher order of Lemma \ref{glucond} for $g=\widetilde{u}_n$ with $d=n-1$ are given by $\frac{\partial^\ell\Gamma_1}{\partial\mu_1^\ell}(0,0)=0,\,\,\, 1\leq \ell\leq n-2.$

\noindent If $n=3,$ then $$\Gamma_1(0,\mu_1)=\frac{\big(m_{1,0}+m_{0,1}\mu_1 +o(\mu_1)\big)^3}{\big(n_{0,1}+n_{2,0}\mu_1+o(\mu_1)\big)^2}+1=\frac{m_{1,0}^2}{n_{0,1}^2}\left(3m_{0,1}-2m_{1,0}\frac{n_{2,0}}{n_{0,1}}\right)\mu_1+o(\mu_1);$$ hence the condition is given by $n_{2,0}=\frac{3m_{0,1}n_{0,1}}{2m_{1,0}}.$

 \medskip

\noindent If $n\geq 4,$ then $\frac{\partial\Gamma_1}{\partial\mu_1}(0,0)=n\frac{m_{1,0}^{n-1}m_{0,1}}{n_{0,1}^{n-1}}.$ Since the coefficient $m_{1,0}$ is nonzero, $m_{0,1}=0.$ This implies $$\Gamma(0,\mu_1)=\frac{n_{2,0}}{n_{0,1}}\mu_1^{n-2}+o(\mu_1^{n-2}),$$ so the last condition is $n_{2,0}=0.$
\end{proof}

\subsection{Remarks on degenerate birational quadratic maps}

\noindent We can do the previous construction also for $n=2$ and find gluing conditions: a germ of biholomorphism $g$ of $\mathbb{C}^2$ around $0$ given by $$g(y,z)=\big(\sum_{0\leq i,j\leq 4} m_{i,j}y^i z^j,\sum_{0\leq i,j\leq 4} n_{i,j}y^i z^j\big)$$ sends $\widehat{\xi}_2$ on $\widehat{\xi}_1$ if and only if
$m_{0,0}=n_{0,0}=0,$ $n_{1,0}=0$ and $m_{1, 0}^2=n_{2, 0}-n_{0, 1}.$

\noindent As we have to blow up $\mathbb{P}^2$ at least ten times to get automorphism with nonzero entropy, we want to find automorphisms~$\varphi$ of $\mathbb{P}^2$ such that $(\varphi\Phi_2)^k\varphi(\widehat{\xi}_2)=\widehat{\xi}_1$ with $k\geq 4$ and $(\varphi\Phi_2)^i\varphi(\widehat{\xi})\not=\widehat{\xi}$ for $0\leq i\leq k-1.$ The Taylor series of~$(\varphi\Phi_2)^k\varphi$ is of the form $$\big(\sum_{0\leq i,j\leq 4} m_{i,j}y^i z^j,\sum_{0\leq i,j\leq 4} n_{i,j}y^i z^j\big)+o\big(\vert\vert(y,z)\vert\vert^4\big)$$ in the affine chart $x=1.$ The degrees of the equations increase exponentially with $k$ so even for $k=4$ it is not easy to explicit a family. However we can verify that if

\begin{align*}
&\varphi=\left[\begin{array}{ccc}
0& 0&-\frac{\alpha^2}{2}\\
0 & 1 & 0 \\
1 & 0 &\alpha
\end{array}
\right] &&\text{with } \alpha \text{ in }\mathbb{C}\text{ such that } \alpha^8+2\alpha^6+4\alpha^4+8\alpha^2+16=0,
\end{align*}

\noindent then $(\varphi\Phi_2)^4\varphi(\widehat{\xi}_2)=\widehat{\xi}_1.$ These examples are conjugate to those studied in \cite{BK4}.

\section{Birational maps whose exceptional locus is a line, II}\label{eg12}

\noindent In this section, we apply the results of \S\ref{casgal} to produce explicit examples of automorphism of rational surfaces obtained from birational maps in the $\mathrm{PGL}(3; \mathbb{C})$-orbit of the $\Phi_n$ for $n \geq 3.$ As we have to blow up $\mathbb{P}^2$ at least ten times to have nonzero entropy, we want to find automorphisms $\varphi$ of $\mathbb{P}^2$ and positive integers $k$ such that

\begin{equation}\label{conditions}
(k+1)(2n-1)\geq 10\, , \, \, \, \, (\varphi\Phi_n)^k\varphi(\widehat{\xi}_2)=\widehat{\xi}_1 \, \, \text{ and }\, \,  (\varphi\Phi_n)^i\varphi( P)\not=P \,\text{ for } \, 0\leq i\leq k-1.
\end{equation}



\subsection{Families of birational maps of degree $n$ with exponential growth conjugate to automorphism of $\mathbb{P}^2$ blown up in~$6n-3$ points}\label{15points}

\noindent Let $\varphi$ be an automorphism of $\mathbb{P}^2.$ We will find solutions of (\ref{conditions}) for $n\geq 3$ and $k=2.$

 \medskip

\noindent Remark that the Taylor series of $(\varphi\Phi_n)^2\varphi$ is of the form $$\big(\sum_{0\leq i,j\leq 2n-2} m_{i,j}y^i z^j,\sum_{0\leq i,j\leq 2n-2} n_{i,j}y^i z^j\big)+o\big(\vert\vert(y,z)\vert\vert^{2n-2}\big)$$ in the affine chart $x=1.$ Assume that $(\varphi\Phi_n)^2\varphi(P)=P;$ one can show that this is the case when
\begin{align*}
&\varphi_{\alpha,\beta}=\left[\begin{array}{ccc}
1 & \gamma &-\frac{1+\delta+\delta^2}{\alpha} \\[1ex]
0 & -1 & 0\\[1ex]
\alpha &\beta &\delta
\end{array}
\right].
\end{align*} One can verify that the conditions of the Proposition \ref{gluglu} are satisfied if $\beta=\dfrac{\alpha\gamma}{2}$ and $(1+\delta)^{3n}=-1.$

\noindent For $0 \leq k \leq 3n-1,$ let $\delta_k=\mathrm{exp} \big(\frac{(2k+1)\mathrm{i}\pi}{3n}\big)-1.$ If \begin{align*}
&\varphi_{\alpha,\beta}=\left[\begin{array}{ccc}
1 & \frac{2\beta}{\alpha} &-\frac{1+\delta_k+\delta_k^2}{\alpha} \\[1ex]
0 & -1 & 0\\[1ex]
\alpha &\beta &\delta_k
\end{array}
\right],
\end{align*}

\noindent  then $(\varphi_{\alpha,\beta}\Phi_n)^2\varphi(\widehat{\xi}_2)=\widehat{\xi}_1.$ Besides, $P \neq \varphi_{\alpha, \beta} \Phi_n (P).$ Hence we get:

\bigskip

\begin{thm}\label{nqccq}
{\sl Assume that $n\geq 3$ and that
\begin{align*}
&\varphi_{\alpha,\beta}=\left[\begin{array}{ccc}
1 & \frac{2\beta}{\alpha} &-\frac{1+\delta_k+\delta_k^2}{\alpha} \\[1ex]
0 & -1 & 0\\[1ex]
\alpha &\beta &\delta_k
\end{array}
\right], && \alpha\in\mathbb{C}^*,\,\beta\in\mathbb{C},\,\delta_k=\mathrm{exp} \left(\frac{(2k+1)\mathrm{i}\pi}{3n}\right)-1,\,0\leq k\leq 3n-1.
\end{align*}

\noindent Each map $\varphi_{\alpha,\beta}\Phi_n$ is conjugate to an automorphism of $\mathbb{P}^2$ blown up in $3(2n-1)$ points.

\noindent The first dynamical degree of~$\varphi_{\alpha,\beta}\Phi_n$ is strictly larger than $1;$ more precisely $\,\lambda(\varphi_{\alpha,\beta}\Phi_n)=\frac{n+\sqrt{n^2-4}}{2}.$

\smallskip
\smallskip

\noindent The family $\varphi_{\alpha,\beta}\Phi_n$ is holomorphically trivial.}
\end{thm}

\begin{proof}
\noindent We only have to prove the two last statements. Let $\varphi$ denote $\varphi_{\alpha,\beta}.$ In the basis
\begin{align*}
&\{\Delta,\,\mathrm{E},\,
\mathrm{F},\,\mathrm{G}^1,\,\ldots,\,\mathrm{G}^{n-2},\,\mathrm{H},\,\mathrm{L}^1,\,\ldots,\,\mathrm{L}^{n-2},\,\varphi\mathrm{E},\,\varphi\mathrm{F},
\,\varphi\mathrm{G}^1,\,\ldots,\varphi\mathrm{G}^{n-2},\,\varphi\mathrm{K},\,\varphi\mathrm{M}^1\,\ldots,\,\varphi\mathrm{M}^{n-2},\\
&\hspace{1cm}\varphi\Phi_n \varphi\mathrm{E},\,\varphi\Phi_n \varphi\mathrm{F},
\,\varphi\Phi_n \varphi\mathrm{G}^1,\,\ldots,\,\varphi\Phi_n\varphi\mathrm{G}^{n-2},\,\varphi\Phi_n \varphi\mathrm{K},\,\varphi\Phi_n \varphi\mathrm{M}^1,\,\ldots,\,\varphi\Phi_n\varphi\mathrm{M}^{n-2}\}\end{align*} the matrix $M$ of $(\varphi\Phi_n)_*$ is
$$\left[\begin{BMAT}{c0c0c0c}{c0c0c0c}
\mathrm{0}_1& \transp A &\mathrm{0}_{1,2n-1} & \mathrm{0}_{1,2n-1}\\
\mathrm{0}_{2n-1,1} & B &\mathrm{0}_{2n-1} & \mathrm{Id}_{2n-1} \\
 A & C & \mathrm{0}_{2n-1} & \mathrm{0}_{2n-1}\\
\mathrm{0}_{2n-1,1} &\mathrm{0}_{2n-1} & \mathrm{Id}_{2n-1}&\mathrm{0}_{2n-1}
\end{BMAT}\right]\in\mathcal{M}_{6n-2}$$

\noindent where

\bigskip

\noindent \begin{footnotesize}\begin{align*}
&A=\left[\begin{array}{cccc}0\\ \vdots \\0\\1\end{array}\right] \in \mathcal{M}_{2n-1,1},&&B=\left[\begin{array}{ccccccc}
0&0&\cdots&\cdots&0&0&1\\
0&0&\cdots&\cdots&0&0&2\\
\vdots&\vdots&&&\vdots&\vdots&\vdots\\
\vdots&\vdots&&&\vdots&\vdots&n\\
\vdots&\vdots&&&\vdots&\vdots&n\\
0&0&\cdots&\cdots&0&0&\vdots\\
0&0&\cdots&\cdots&0&0&n
\end{array}\right]\in\mathcal{M}_{2n-1}, &&
C=\left[\begin{array}{ccccccc}
1&0&\cdots&\cdots&0&0&-1\\
0&0&\cdots&\cdots&0&1&-2\\
\vdots&\vdots&&\iddots&\iddots&0&\vdots\\
\vdots&\vdots&\iddots&\iddots&\iddots&\vdots&-n\\
0&0&\iddots&\iddots&&\vdots&-n\\
0&1&0&\cdots&\cdots&0&\vdots\\
0&0&0&\cdots&\cdots&0&-n
\end{array}\right]\in\mathcal{M}_{2n-1}.
\end{align*}
\end{footnotesize}

\noindent Its characteristic polynomial is $(X^2-nX+1)(X^2-X+1)^{n-2}(X+1)^{n-1}(X^2+X+1)^n  (X-1)^{n+1}.$ Hence $$\lambda(\varphi\Phi_n)=\frac{n+\sqrt{n^2-4}}{2}$$ which is larger than $1$ as soon as $n\geq 3.$

\medskip

\noindent Fix a point $(\alpha_0, \beta_0)$ in $\mathbb{C}^*\times\mathbb{C}.$
 We can find  locally around $(\alpha_0, \beta_0)$ a matrix $M_{\alpha,\beta}$ depending holomorphically on $(\alpha, \beta)$ such that for all $(\alpha,\beta)$ near $(\alpha_0,\beta_0),$ we have $\varphi_{\alpha,\beta}\Phi_n= M_{\alpha,\beta}^{-1}\varphi_{\alpha_0,\beta_0}\Phi_nM_{\alpha,\beta}:$
if $\mu$ is a local holomorphic solution of the equation $\alpha=\mu^n\alpha_0$ such that $\mu_0=1$ we can take
$$M_{\alpha,\beta}=\left[\begin{array}{ccc} 1 & \frac{\beta-\beta_0 \mu}{\mu^n\alpha_0} & 0\\[1 ex] 0 &\frac{1}{\mu^{n-1}}& 0 \\[1 ex] 0 & 0 & \frac{1}{\mu^n}\end{array}\right].$$
\noindent Thus $\mathfrak{m}(\varphi_{\alpha, \beta})=0.$

\end{proof}

\begin{rem}
Assume that $\delta_k=-2$ and $n$ is odd. Consider the automorphism $A$ of $\mathbb{P}^2$ given by
\begin{align*}
&A=(uy:\alpha x+\beta y-z:z), && \alpha\in\mathbb{C}^*,\,\beta\in\mathbb{C},\,u^n=\alpha.
\end{align*}

\noindent One can verify that $A(\varphi_{\alpha,\beta}\Phi_n)A^{-1}=(xz^{n-1}:z^n:x^n+z^n-yz^{n-1})$ which is of the form of (\ref{bedfordkk}).
\end{rem}

\subsection{Families of birational maps of degree $n$ with exponential growth conjugate to an automorphism of $\mathbb{P}^2$ blown up in $4n-2$ points}\label{10points}

\noindent In this section we will assume that $n$ is larger than $4.$ In that case,
we succeed in providing solutions of (\ref{conditions}) for $k=1.$

\begin{thm}\label{nqcq2}
{\sl Assume that $n\geq 4$ and
\begin{align*}
&\varphi_{\alpha,\beta,\gamma,\delta}=\left[\begin{array}{ccc}
\alpha &\beta & \frac{\beta(\gamma^2 \varepsilon_k-\alpha^2)}{\delta(\alpha-\gamma)}\\
0 & \gamma & 0\\
 \frac{\delta( \alpha-\gamma)}{\beta} & \delta &-\alpha
 \end{array}
 \right], &&\alpha,\, \beta \in \mathbb{C},\,\gamma,\,\delta\in\mathbb{C}^*,\,\alpha\not=\gamma,\,
 \varepsilon_k=\mathrm{exp} \, \left(\frac{(2k+1)\mathrm{i}\pi}{n}\right),\,0\leq k\leq n-1.
\end{align*}

\noindent Each map $\varphi_{\alpha,\beta,\gamma,\delta}\Phi_n$ is conjugate to an automorphism of $\mathbb{P}^2$ blown up in $4n-2$ points.

\noindent The first dynamical degree of~$\varphi_{\alpha,\beta,\gamma,\delta}\Phi_n$ is strictly larger than $1;$ more precisely $\lambda(\varphi_{\alpha,\beta,\gamma,\delta}\Phi_n)=\frac{(n-1)+\sqrt{(n-1)^2-4}}{2}.$

\smallskip
\smallskip

\noindent The family $\varphi_{\alpha,\beta,\gamma,\delta}\Phi_n$ is holomorphically trivial.}
\end{thm}

\begin{proof}
\noindent The first point is again a consequence of Proposition \ref{gluglu}.
\par
\smallskip
\noindent Let $\varphi$ denote $\varphi_{\alpha,\beta,\gamma,\delta}.$
In the basis
$$\{\Delta,\,\mathrm{E},\,
\mathrm{F},\,\mathrm{G}^1,\,\ldots,\,\mathrm{G}^{n-2},\,\mathrm{H},\,\mathrm{L}^1,\,\ldots,\,\mathrm{L}^{n-2},\,\varphi\mathrm{E},\,\varphi\mathrm{F},
\,\varphi\mathrm{G}^1,\,\ldots,\varphi\mathrm{G}^{n-2},\,\varphi\mathrm{K},\,\varphi\mathrm{M}^1\,\ldots,\,\varphi\mathrm{M}^{n-2}\}$$ the matrix $M$ of $(\varphi\Phi_n)_*$ is
$$\left[\begin{BMAT}{c0c0c}{c0c0c}
\mathrm{0}_1& \transp A & \mathrm{0}_{1,2n-1}\\
\mathrm{0}_{2n-1,1} & B & \mathrm{Id}_{2n-1} \\
 A & C & \mathrm{0}_{2n-1}\\
\end{BMAT}\right]\in\mathcal{M}_{4n-1}$$

\noindent where

\bigskip

\noindent \begin{footnotesize}\begin{align*}
&A=\left[\begin{array}{cccc}0\\ \vdots \\0\\1\end{array}\right] \in \mathcal{M}_{2n-1,1},&&B=\left[\begin{array}{ccccccc}
0&0&\cdots&\cdots&0&0&1\\
0&0&\cdots&\cdots&0&0&2\\
\vdots&\vdots&&&\vdots&\vdots&\vdots\\
\vdots&\vdots&&&\vdots&\vdots&n\\
\vdots&\vdots&&&\vdots&\vdots&n\\
0&0&\cdots&\cdots&0&0&\vdots\\
0&0&\cdots&\cdots&0&0&n
\end{array}\right]\in\mathcal{M}_{2n-1}, &&
C=\left[\begin{array}{ccccccc}
1&0&\cdots&\cdots&0&0&-1\\
0&0&\cdots&\cdots&0&1&-2\\
\vdots&\vdots&&\iddots&\iddots&0&\vdots\\
\vdots&\vdots&\iddots&\iddots&\iddots&\vdots&-n\\
0&0&\iddots&\iddots&&\vdots&-n\\
0&1&0&\cdots&\cdots&0&\vdots\\
0&0&0&\cdots&\cdots&0&-n
\end{array}\right]\in\mathcal{M}_{2n-1}.
\end{align*}
\end{footnotesize}

\noindent Its characteristic polynomial is $(X^2-(n-1)X+1)(X^2+1)^{n-2}(X+1)^{n-1}(X-1)^{n+1}.$ Hence $$\lambda(\varphi\Phi_n)=\frac{(n-1)+\sqrt{(n-1)^2-4}}{2}$$ which is larger than $1$ as soon as $n\geq 4.$

\medskip

\noindent Fix a point $(\alpha_0, \beta_0, \gamma_0, \delta_0)$ in $\mathbb{C}\times\mathbb{C}\times\mathbb{C}^*\times\mathbb{C}^*$ such that $\alpha_0\not=\gamma_0.$
 We can find locally around $(\alpha_0,\beta_0,\gamma_0,\delta_0)$ a ma\-trix~$M_{\alpha,\beta,\gamma,\delta}$ depending holomorphically on $(\alpha,\beta,\gamma,\delta)$ such that for all $(\alpha,\beta,\gamma,\delta)$ near $(\alpha_0,\beta_0,\gamma_0,\delta_0),$ we have $$\varphi_{\alpha,\beta,\gamma,\delta}\Phi_n= M_{\alpha,\beta,\gamma,\delta}^{-1}\varphi_{\alpha_0,\beta_0,\gamma_0,\delta_0}\Phi_nM_{\alpha,\beta,\gamma,\delta}:$$ if $\mu$ is a local holomorphic solution of the equation $\beta=\dfrac{\mu^n\beta_0 \gamma_0 \, \delta (\gamma-\alpha)}{\gamma \delta_0(\gamma_0-\alpha_0)}$ such that $\mu_0=0,$ we can take
\begin{align*}
&M_{\alpha,\beta,\gamma,\delta}=\left[\begin{array}{ccc} 1 & A & B\\ 0 &\mu^{n-1}& 0 \\ 0 & 0 & \mu^n\end{array}\right], &&\text{where} \quad \quad  A=\frac{\beta_0 \mu^{n-1}(\gamma \delta_0-\mu\gamma_0 \, \delta)}{\gamma\delta_0 (\gamma_0-\alpha_0)} &&\text{and} \quad \quad B=\frac{\beta_0 \mu^n(\alpha_0 \gamma-\alpha \gamma_0)}{\gamma\delta_0 (\alpha_0-\gamma_0)}.
\end{align*}
\end{proof}

\begin{rem}
If $u$ is any solution of the equation $u^n=(\alpha-\gamma)\varepsilon_k^n\gamma^{n-1}\beta^{n-1}\delta,$ consider the automorphism $A$ of $\mathbb{P}^2$ given~by
\begin{align*}
&A=(uy:\delta(\alpha-\gamma) x+\beta\delta y-\alpha\beta z:\varepsilon_k\beta\gamma z), && \alpha,\,\beta\in\mathbb{C},\,\gamma,\,\delta\in\mathbb{C}^*,\,\alpha\not=\gamma.
\end{align*}

\noindent One can verify that $A(\varphi_{\alpha,\beta,\gamma,\delta}\Phi_n)A^{-1}=(xz^{n-1}:z^n:x^n+\varepsilon_k yz^{n-1})$ which is of the form of (\ref{bedfordkk}).
\end{rem}

\subsection{An example in degree $3$ with no invariant line}

\begin{thm}\label{nqcq3}
{\sl Let $\varphi_\alpha$ be the automorphism of the complex projective plane given by
\begin{align*}
&\varphi_\alpha=\left[\begin{array}{ccc}\alpha & 2(1-\alpha) & 2+\alpha-\alpha^2\\
-1 & 0 & \alpha+1 \\
1 & -2 & 1-\alpha
\end{array} \right], &&\alpha\in\mathbb{C}\setminus\{0,\,1\}.
\end{align*}

\noindent Each map $\varphi_\alpha\Phi_3$ has no invariant line and is conjugate to an automorphism of $\mathbb{P}^2$ blown up in $15$ points.

\noindent The first dynamical degree of $\varphi_\alpha\Phi_3$ is $\frac{3+\sqrt{5}}{2}>1.$

\smallskip

\noindent The family $\varphi_\alpha\Phi_3$ is holomorphically trivial.}
\end{thm}

\begin{rem}\label{malin}
The three points $P,$ $\varphi_\alpha(P)$ and $\varphi_\alpha\Phi_3\varphi_\alpha(P)$ are not aligned in the complex projective plane. Indeed,
$P=(1:0:0),$ $\varphi_\alpha(P)=(\alpha:-1:1)$ and $\varphi_\alpha\Phi_3\varphi_\alpha(P)=( \alpha:1:1).$
\end{rem}

\begin{proof}
\noindent The first assertion is given by Remark \ref{malin} and by Proposition \ref{gluglu}, and the second by Theorem \ref{nqccq}.

\medskip

\noindent Fix a point $\alpha_0$ in $\mathbb{C}\setminus\{0,\,1\}.$ We can find  locally around $\alpha_0$ a matrix $M_\alpha$ depending holomorphically on $\alpha$ such that for all~$\alpha$ near $\alpha_0,$ we have $\varphi_\alpha\Phi_3= M_\alpha^{-1}\varphi_{\alpha_0}\Phi_3M_\alpha:$ it suffices to take
$$M_\alpha=\left[\begin{array}{ccc} 1 & 0 & \alpha_0-\alpha\\ 0 &1& 0 \\ 0 & 0 & 1\end{array}\right].$$
\end{proof}

\subsection{A conjecture}\label{conjecture}

\noindent Let us recall a question which was communicated to the first author by E. Bedford:
\medskip

\noindent {\sl Does there exist a birational map of the projective plane $f$ such that for all $\varphi$ in~$\mathrm{PGL}(3;\mathbb{C}),$ the map $\varphi f$ is not birationally conjugate to an automorphism with positive entropy?}

\medskip

\noindent We do not know at the present time the answer to this question. However, after a long series of examples, it seems that the birational maps $\Phi_n$ satisfy a rigidity property:

\begin{conj}
{\sl Let U be an open set of $\mathbb{C}^d,$ $n$ be an integer greater than or equal to $3$ and $\varphi_{\alpha}$ be a holomorphic family of matrices in $\mathrm{PGL}(3;\mathbb{C})$ parameterized by $U.$ Assume that there exists a positive integer $k$ such that
$$(k+1)(2n-1)\geq 10, \, \, \, (\varphi_\alpha\Phi_n)^i\varphi_\alpha( P)\not=P \,\text{ for } \, 0\leq i\leq k-1 \, \, \text{and} \, \, \,  (\varphi_\alpha\Phi_n)^k\varphi_\alpha(\widehat{\xi}_2)=\widehat{\xi}_1.$$
Then $(\varphi_{\alpha}\Phi_n)_{\alpha\in U}$ is holomorphically trivial.}
\end{conj}

\noindent Let us remark that the maps of the form (\ref{bedfordkk}) don't satisfy this conjecture: for $n=5$ one can verify that for any nonzero complex number $s,$
\begin{align*}
& Af_a=f_{s^2a}B, &&\text{where} \quad \quad \quad  A= \big(x:\frac{y}{s}:c(1/s-s^4)y+s^4z\big) \quad \text{and} \quad  B=(sx:s^5y:z).
\end{align*}

\noindent Thus $f_{s^2a}$ and $B^{-1}Af_a$ are linearly conjugate. For a fixed nonzero complex number $a,$ if $\varphi_s=B^{-1}A,$ we consider the family $(\varphi_sf_a)_s.$ This family can be lifted to a family of rational surface automorphisms. Since the familiy $f_b$ is generically effective (\textit{cf.} Example \ref{malin}), the generic number of parameters of $(\varphi_sf_a)_s$  is $1.$

\section{A birational cubic map blowing down one conic and one line}\label{eg2}

\noindent Let $f$ denote the following birational map $$f=\big(y^2z:x(xz+y^2):y(xz+y^2)\big);$$ it blows up two points and blows down two curves, more precisely
\begin{align*}
& \mathrm{Ind}\, f=\{R=(1:0:0),\,P=(0:0:1)\},&& \mathrm{Exc}\, f=\big(\mathcal{C}=\{xz+y^2=0\}\big)\cup\big(\Delta'=\{y=0\}\big).
\end{align*}

\noindent One can verify that $f^{-1}=\big(y(z^2-xy):z(z^2-xy):xz^2\big)$ and
\begin{align*}
& \mathrm{Ind}\, f^{-1}=\{Q=(0:1:0),\,R\},&& \mathrm{Exc}\, f^{-1}=\big(\mathcal{C}'=\{z^2-xy=0\}\big)\cup\big(\Delta''=\{z=0\}\big).
\end{align*}

\noindent Set $\Delta=\{x=0\}.$ The sequences of blowups corresponding to the minimal desingularization of $f$ can be computed in five steps, as explained below:

\begin{itemize}
\item First we blow up $R$ in the domain and in the range and denote by
$\mathrm{E}$ the exceptional divisor. One can show that $\mathcal{C}_1=\{u_1+v_1=0\}$ is sent on $\mathrm{E},$ $\mathrm{E}$ is blown down to $Q=(0:1:0)$ and $S=\mathrm{E}\cap\Delta''_1$ is a point of indeterminacy.

\item Next we blow up $P$ in the domain and $Q$ in the range and denote by $\mathrm{F}$ (resp. $\mathrm{G}$) the exceptional divisor associated with $P$ (resp. $Q$). One can verify that $\mathrm{F}$ is sent on $\mathcal{C}'_2,$ $\mathrm{E}_1$ is blown down to $T=\mathrm{G}\cap\Delta_2$ and $\Delta'_2$ is blown down to $T.$

\item Then we blow up $S$ in the domain and $T$ in the range and denote by $\mathrm{H}$ (resp. $\mathrm{K}$) the exceptional divisor obtained by blowing up $S$ (resp. $T$). One can show that
$\mathrm{H}$ is sent on $\mathrm{K};$ $\mathrm{E}_2$ and $\Delta'_3$ are blown down to a point $V$ on $\mathrm{K},$ and there is a point $U$ of indeterminacy on $\mathrm{H}.$

\item We will now blow up $U$ in the domain and $V$ in the range; let $\mathrm{L}$ (resp. $\mathrm{M}$) be the exceptional divisor obtained by blowing up $U$ (resp. $V$).
There is a point $Y$ of indeterminacy on $\mathrm{L},$ $\mathrm{L}$ is sent on $\mathrm{G}_2,$  $\mathrm{E}_3$ is sent on $\mathrm{M}$ and $\Delta'_4$ is blown down to a point $Z$ of $\mathrm{M}.$

\item Finally we blow up $Y$ in the domain and $Z$ in the range. The line $\Delta'_5$ is sent on $\Omega$ and $\mathrm{N}$ is sent on $\Delta''_5,$ where $\Omega$ (resp. $\mathrm{N}$) is the exceptional divisor obtained by blowing up $Z$ (resp. $Y$).
\end{itemize}

\bigskip

\begin{pro}\label{coniquecomp}
{\sl Let $\widehat{\zeta}_1$ (resp. $\widehat{\zeta}_2$) denote the point infinitely near $R$ (resp. $Q$) obtained by blowing up $R,$ $S,$ $U$ and $Y$ (resp. $Q,$ $T,$ $V$ and $Z$). If we put $\,\widehat{\xi}_1=\widehat{\zeta}_1\cup\{P\}$ and $\widehat{\xi}_1=\widehat{\zeta}_2\cup\{R\},$ then $\widehat{\xi}_1$ and $\widehat{\xi}_2$ correspond to the minimal desingularization of $f.$ The map $f$ induces an isomorphism between $\mathrm{Bl}_{\widehat{\xi}_1}\,\mathbb{P}^2$ and $\mathrm{Bl}_{\widehat{\xi}_2}\,\mathbb{P}^2,$ and the different components are swapped as follows:
\begin{align*}
&&\mathcal{C}\to\mathrm{E}, &&\mathrm{F}\to\mathcal{C}', && \mathrm{H}\to\mathrm{K}, &&\mathrm{L}\to\mathrm{G}, && \mathrm{E}\to\mathrm{M}, &&\Delta'\to\Omega,   && \mathrm{N}\to \Delta''.
\end{align*}}
\end{pro}

\noindent The following statement gives the gluing conditions:

\begin{pro}\label{recolconfig9}
{\sl Let $u(x,z)=\left(\displaystyle\sum_{(i,j)\in\mathbb{N}^2} m_{i,j}x^iz^j, \displaystyle\sum_{ (i,j)\in\mathbb{N}^2} n_{i,j}x^iz^j\right)$ be a germ of biholomorphism at $Q.$

\noindent Then $u$ can be lifted to a germ of biholomorphism between $\mathrm{Bl}_{\widehat{\zeta}_2}\mathbb{P}^2$ and  $\mathrm{Bl}_{\widehat{\zeta}_1}\mathbb{P}^2$  if and only if:
\begin{itemize}
\item $m_{0,0}=n_{0,0}=0;$

\item $n_{0,1}=0;$

\item $n_{0,2}+n_{1,0}+m_{0,1}^2=0;$

\item $n_{0,3}+n_{1,1}+2m_{0,1}(m_{0,2}+m_{1,0})=0.$
\end{itemize}}
\end{pro}

\noindent Let $\varphi$ be an automorphism of $\mathbb{P}^2.$ We will adjust $\varphi$ in order that $(\varphi f)^k\varphi$ sends $\widehat{\xi}_2$ onto $\widehat{\xi}_1.$ As we have to blow up~$\mathbb{P}^2$ at least ten times to have nonzero entropy, $k$ must be larger than $2,$ $\{ \widehat{\xi}_1,$ $\varphi\widehat{\xi}_2,$ $\varphi f\varphi\widehat{\xi}_2,$ $(\varphi f)^2\varphi\widehat{\xi}_2,$ $\ldots,$ $(\varphi f)^{k-1}\varphi\widehat{\xi}_2 \}$ must all have distinct supports and $(\varphi f)^k\varphi\widehat{\xi}_2=\widehat{\xi}_1.$ We provide such matrices for $k=3:$ by Proposition~\ref{recolconfig9} one can verify that for every nonzero complex number $\alpha,$
$$\varphi_\alpha=\left[\begin{array}{ccc}\frac{2\alpha^3}{343}(37\mathrm{i}\sqrt{3}+3)& \alpha& -\frac{2\alpha^{2}}{49}(5\mathrm{i}\sqrt{3}+11)\\[2ex]
\frac{\alpha^2}{49}(-15+11\mathrm{i}\sqrt{3}) & 1 & -\frac{\alpha}{14}(5\mathrm{i}\sqrt{3}+11)\\[2ex]
 -\frac{\alpha}{7}(2\mathrm{i}\sqrt{3}+3)& 0 & 0\end{array} \right]$$

\noindent is such a $\varphi.$

\begin{thm}\label{conique}
{\sl Assume that $f=\big(y^2z:x(xz+y^2):y(xz+y^2)\big)$ and that
\begin{align*}
\varphi_\alpha=\left[\begin{array}{ccc}\frac{2\alpha^3}{343}(37\mathrm{i}\sqrt{3}+3)& \alpha& -\frac{2\alpha^{2}}{49}(5\mathrm{i}\sqrt{3}+11)\\[2ex]
\frac{\alpha^2}{49}(-15+11\mathrm{i}\sqrt{3}) & 1 & -\frac{\alpha}{14}(5\mathrm{i}\sqrt{3}+11)\\[2ex]
 -\frac{\alpha}{7}(2\mathrm{i}\sqrt{3}+3)& 0 & 0\end{array} \right], && \alpha \in \mathbb{C}^*.
\end{align*}

\noindent Each map $\varphi_\alpha f$ is conjugate to an automorphism of $\mathbb{P}^2$ blown up in $15$ points.

\noindent The first dynamical degree of~$\varphi_\alpha f$ is $\lambda(\varphi_\alpha f)=\frac{3+\sqrt{5}}{2}.$

\smallskip

\noindent The family $\varphi_\alpha f$ is holomorphically trivial.}
\end{thm}

\begin{proof}
\noindent Set $\varphi=\varphi_\alpha.$ In the basis
\begin{align*}
&\{\Delta',\,\mathrm{E},\,
\mathrm{F},\,\mathrm{H},\,\mathrm{L},\,\mathrm{N},\,\varphi \mathrm{E},\,\varphi \mathrm{G},
\,\varphi \mathrm{K},\,\varphi \mathrm{M},\,\varphi \Omega,\,\varphi f\varphi\mathrm{E},\,\varphi f\varphi\mathrm{G},
\,\varphi f\varphi\mathrm{K},\,\varphi f\varphi\mathrm{M},\,\varphi f\varphi\Omega\}\end{align*} the matrix $M$ of $(\varphi f)_*$ is
$$\left[\begin{array}{cccccccccccccccc}
0 & 0 & 2 & 0 & 0 & 1 & 0 & 0 & 0 & 0 & 0 & 0 & 0 & 0 & 0 & 0\\
0 & 0 & 2 & 0 & 0 & 1 & 0 & 0 & 0 & 0 & 0 & 0 & 1 & 0 & 0 & 0\\
0 & 0 & 2 & 0 & 0 & 1 & 0 & 0 & 0 & 0 & 0 & 1 & 0 & 0 & 0 & 0\\
0 & 0 & 2 & 0 & 0 & 1 & 0 & 0 & 0 & 0 & 0 & 0 & 0 & 1 & 0 & 0\\
0 & 0 & 2 & 0 & 0 & 1 & 0 & 0 & 0 & 0 & 0 & 0 & 0 & 0 & 1 & 0\\
0 & 0 & 2 & 0 & 0 & 1 & 0 & 0 & 0 & 0 & 0 & 0 & 0 & 0 & 0 & 1\\
0 & 0 & -1 & 0 & 0 & -1 & 0 & 0 & 0 & 0 & 0 & 0 & 0 & 0 & 0 & 0\\
0 & 0 & -1 & 0 & 1 & -1 & 0 & 0 & 0 & 0 & 0 & 0 & 0 & 0 & 0 & 0\\
0 & 0 & -2 & 1 & 0 & -1 & 0 & 0 & 0 & 0 & 0 & 0 & 0 & 0 & 0 & 0\\
0 & 1 & -3 & 0 & 0 & -1 & 0 & 0 & 0 & 0 & 0 & 0 & 0 & 0 & 0 & 0\\
1 & 0 & -4 & 0 & 0 & -1 & 0 & 0 & 0 & 0 & 0 & 0 & 0 & 0 & 0 & 0\\
0 & 0 & 0 & 0 & 0 & 0 & 1 & 0 & 0 & 0 & 0 & 0 & 0 & 0 & 0 & 0\\
0 & 0 & 0 & 0 & 0 & 0 & 0 & 1 & 0 & 0 & 0 & 0 & 0 & 0 & 0 & 0\\
0 & 0 & 0 & 0 & 0 & 0 & 0 & 0 & 1 & 0 & 0 & 0 & 0 & 0 & 0 & 0\\
0 & 0 & 0 & 0 & 0 & 0 & 0 & 0 & 0 & 1 & 0 & 0 & 0 & 0 & 0 & 0\\
0 & 0 & 0 & 0 & 0 & 0 & 0 & 0 & 0 & 0 & 1 & 0 & 0 & 0 & 0 & 0
\end{array}\right].$$

\noindent Its characteristic polynomial is $(X-1)^4 (X+1)^2 (X^2-X+1) (X^2+X+1)^3 (X^2-3X+1).$ Hence $\lambda(\varphi f)=\frac{3+\sqrt{5}}{2}.$

\medskip

\noindent Fix a point $\alpha_0$ in $\mathbb{C}^*.$
We can find  locally around $\alpha_0$ a matrix $M_\alpha$ depending holomorphically on $\alpha$ such that for all $\alpha$ near $\alpha_0,$ we have $\varphi_\alpha f= M_{\alpha}^{-1}\varphi_{\alpha_0}fM_{\alpha}:$ take
$$M_\alpha=\left[\begin{array}{ccc} 1 & 0 & 0\\ 0 &\frac{\alpha}{\alpha_0}& 0 \\ 0 & 0 &\frac{\alpha^2}{\alpha_0^2}\end{array}\right].$$ This implies that $\varphi_\alpha f$ is holomorphically trivial.
\end{proof}

\section{Families of rational surfaces}\label{deformations}

\noindent Families of rational surfaces are usually constructed by blowing up $\mathbb{P}^2$ (or a Hirzebruch surface $\mathbb{F}_n$) successively at $N$ points $p_1, \ldots, p_N$ and then by deforming the points $p_i.$ Such deformations can be holomorphically trivial: the simplest example is given by the family $\mathrm{Bl}_{M_t p_1,\ldots,M_t p_N} \mathbb{P}^2,$ where $p_1, \ldots, p_N$ are $N$ distinct points in $\mathbb{P}^2$ and~$t \mapsto M_t$ is a holomorphic curve in $\mathrm{PGL}(3;\mathbb{C})$ such that $M_0=\mathrm{Id}.$ In this section, we give a general description of deformations of rational surfaces, using the general theory of Kodaira and Spencer (\cite{Kodaira}). Then, after a general digression about the generic numbers of parameters of an algebraic deformation, we will give a practical way to count the generic number of parameters of a given family of rational surfaces with no holomorphic vector field. As an application, for any family of birational maps which can be lifted to a family of rational surface automorphisms, we compare the generic number of parameters of this family (as defined in \S\ref{jabuse}) and the generic number of parameters of the associated family of rational surfaces.

\smallskip

\noindent This section can be read independently from the other ones (except \S\ref{nonbasic}), its aim is to provide in some specific cases a geometric interpretation of the generic number of parameters for families of Cremona transformations introduced in \S\ref{jabuse}.

\subsection{Deformations of basic rational surfaces} \label{def}

\par \medskip
\noindent Recall that every rational surface can be obtained by blowing up finitely many times $\mathbb{P}^2$ of a Hirzebruch surface~$\mathbb{F}_n$ (\emph{see} \cite[p. \!520]{G-H}). A rational surface is called \textit{basic} if it is a blowup of~$\mathbb{P}^2.$
By \cite[Th. 5]{Nag}, if $f$ is an automorphism of a rational surface $X$ such that $f^{*}$ is of infinite order on~$\mathrm{Pic}(X),$ then~$X$ is basic. Furthermore, by the main result of \cite{Harbourne}, $X$ carries no nonzero holomorphic vector field.
\par \medskip

\noindent For each integer $N,$ let us define a sequence of deformations $\pi_N\colon\mathfrak{X}_N \to S_N$ as follows:
\begin{itemize}

\item $S_0$ is a point and $\mathfrak{X}_0=\mathbb{P}^2.$

\item $S_{N+1}=\mathfrak{X}_N;$ $\mathfrak{X}_{N+1}=\mathrm{Bl}_{\mathfrak{X}_N}(\mathfrak{X}_N \times_{S_N} \mathfrak{X}_N),$ where $\mathfrak{X}_N$ is diagonally embedded in $\mathfrak{X}_N \times_{S_N} \mathfrak{X}_N;$ and $\pi_{N+1}$ is obtained by composing the blow up morphism from $\mathfrak{X}_{N+1}$ to $\mathfrak{X}_N \times_{S_N} \mathfrak{X}_N$ with the first projection.

\end{itemize}

\par\medskip

\noindent The varieties $S_N$ and $\mathfrak{X}_N$ are smooth and projective, they can be given the following geometric interpretation:
\begin{itemize}
\item For $N \geq 1,$ $S_N$ is the set of ordered lists of (possibly infinitely near) points of $\mathbb{P}^2$ of length $N.$ This means that $$S_N=\{p_1, \ldots, p_N \, \, \, \text{such that}\, \, \, p_1 \in \mathbb{P}^2\, \,  \text{and if} \, \,  2 \leq i \leq N, \,  p_i \in \mathrm{Bl}_{p_{i-1}} \mathrm{Bl}_{p_{i-2}} \ldots \mathrm{Bl}_{p_{1}} \mathbb{P}^2 \}.$$ Elements of $S_N$ will be denoted by $\widehat{\xi}.$

\item If $N \geq 1,$ $\mathfrak{X}_N$ is the universal family of rational surfaces over $S_N:$ for every $\widehat{\xi}$ in $S_N,$ the fiber $\pi_N^{-1}(\widehat{\xi}\,)$ of $\widehat{\xi}$ in $\mathfrak{X}_N$ is the rational surface $\mathrm{Bl}_{\widehat{\xi}}\, \mathbb{P}^2$ parameterized by $\widehat{\xi}.$
\end{itemize}

\par\medskip

\noindent The group $\mathrm{PGL}(3;\mathbb{C})$ of biholomorphisms of $\mathbb{P}^2$ acts naturally on the configuration spaces $S_N:$ if $g$ is an element of~$\mathrm{PGL}(3;\mathbb{C})$ and $\widehat{\xi}$ lies in $S_N,$ $g.\widehat{\xi}$ is the unique element of $S_N$ such that $g$ induces an isomor\-phism between $\mathrm{Bl}_{\widehat{\xi}}\, \mathbb{P}^2$ and~$\mathrm{Bl}_{g.\widehat{\xi}}\, \mathbb{P}^2.$ Then we have an easy but important fact:

\begin{lem}\label{idcomp}
{\sl Let $N$ be a positive integer, $\widehat{\xi}$ be an element of $S_N$ and $G_{\widehat{\xi}}$ be the stabilizer of $\,\widehat{\xi}$ in $\mathrm{PGL(3;\mathbb{C})}.$ Then the identity components of $G_{\widehat{\xi}}$ and of $\mathrm{Aut}(\mathrm{Bl}_{\widehat{\xi}} \mathbb{P}^2)$ are canonically isomorphic. In particular, the Lie algebra of holomorphic vector fields on $\mathrm{Bl}_{\widehat{\xi}}\mathbb{P}^2$ is canonically isomorphic to the Lie algebra of $G_{\widehat{\xi}}.$}
\end{lem}

\begin{proof}
\noindent The group $G_{\widehat{\xi}}$ is clearly a subgroup of $\mathrm{Aut}(\mathrm{Bl}_{\widehat{\xi}} \mathbb{P}^2)$. We write $\widehat{\xi}=\widehat{\xi}' \cup \{p\},$ where $\widehat{\xi}'$ is in $S_{N-1}$ and $p$ is in~$\mathrm{Bl}_{\widehat{\xi}'}\mathbb{P}^2;$ and we denote by $E$ the exceptional divisor of the blowup of $\mathrm{Bl}_{\widehat{\xi}'}\mathbb{P}^2$ at $p.$ Then for any $u$ in $\mathrm{Aut}(\mathrm{Bl}_{\widehat{\xi}}\mathbb{P}^2),$ the intersection number $E\cdot u(E)$ depends only on the connected component of $u$ in the automorphism group of $\mathrm{Bl}_{\widehat{\xi}}\mathbb{P}^2.$ In particular, if $u$ is in the identity component of this group, $E\cdot u(E)=E\cdot E=-1.$ Since $u(E)$ is an irreducible curve on $\mathrm{Bl}_{\widehat{\xi}}\mathbb{P}^2,$ this implies that $u(E)=E$ (otherwise the intersection number $E\cdot u(E)$ would be nonnegative) so that $u$ is induced by an automorphism of $\mathrm{Bl}_{\widehat{\xi}'}\mathbb{P}^2.$ Therefore, if $j \colon \mathrm{Aut}(\mathrm{Bl}_{\widehat{\xi}'} \mathbb{P}^2) \rightarrow \mathrm{Aut}(\mathrm{Bl}_{\widehat{\xi}} \mathbb{P}^2)$ is the natural injection, the image of $j$ contains the identity component of $\mathrm{Aut}(\mathrm{Bl}_{\widehat{\xi}} \mathbb{P}^2).$ It follows that $j$ induces an isomorphism between the identity component of $\mathrm{Aut}(\mathrm{Bl}_{\widehat{\xi}'} \mathbb{P}^2)$ and the identity component of $\mathrm{Aut}(\mathrm{Bl}_{\widehat{\xi}} \mathbb{P}^2).$ By applying repeatedly this argument, we get that the identity component of $G_{\widehat{\xi}}$ and of $\mathrm{Aut}(\mathrm{Bl}_{\widehat{\xi}} \mathbb{P}^2)$ are isomorphic.
\end{proof}

\noindent In the sequel, for every integer $N \geq 4,$ we will denote by $S_N^{\dag}$ the Zariski-dense open subset of $S_N$ consisting of points~$\widehat{\xi}$ in $S_N$ such that $G_{\widehat{\xi}}$ is trivial. The associated rational surfaces $\{\mathrm{Bl}_{\widehat{\xi}} \, \mathbb{P}^2, \, \widehat{\xi} \in S_N^{\dag}\}$ are rational surfaces in the family $\mathfrak{X}_N$ carrying no nonzero holomorphic vector field. Besides, the action of $\mathrm{PGL(3;\mathbb{C})}$ defines a regular foliation on $S_N^{\dag}.$

\par\medskip
\noindent For any point $\widehat{\xi}$ in $S_N,$ let $O_{\widehat{\xi}}$ be the $\mathrm{PGL(3;\mathbb{C})}$-orbit of $\widehat{P}$ in $S_N.$ The main result of this section is:

\begin{thm} \label{easy}
{\sl Let $N$ be a positive integer. For any point $\widehat{\xi}$ in $S_N,$ the Kodaira-Spencer map of $\mathfrak{X}_N$ at $\widehat{\xi}$ is surjective and its kernel is equal to $\mathrm{T}_{\widehat{\xi}} \, O_{\widehat{\xi}}.$}
\end{thm}

\noindent Before giving the proof, we start by some generalities. Let $(\mathfrak{X}, \pi, B)$ be a deformation and $b$ be a point in $B.$ Recall that~$\mathfrak{X}$ is complete at $b$ if any small deformation of $\mathfrak{X}_b$ is locally induced by $\mathfrak{X}$ via a holomorphic map. Let us quote two fundamental results in deformation theory (\emph{see} \cite[p. \!$270$ and $284$]{Kodaira}):

\begin{enumerate}
\item [(i)] \textit{Theorem of existence.} Let $X$ be a complex compact manifold such that $\mathrm{H}^{2}(X, \mathrm{T}X)=0.$ Then there exists a deformation $(\mathfrak
{X}, \pi, B)$ of $X$ such that $\mathfrak{X}_0=X$ and $\mathrm{KS}_0(\mathfrak{X}):\mathrm{T}_0 B \rightarrow \mathrm{H}^1(X, \mathrm{T}X)$ is an isomorphism.

\item [(ii)] \textit{Theorem of completeness.} Let $(\mathfrak{X}, \pi, B)$ be a deformation and $b$ be in $B$ such that $\mathrm{KS}_b(\mathfrak{X}) \colon \mathrm{T}_b B \rightarrow \mathrm{H}^1(\mathfrak{X}_b, \mathrm{T}\mathfrak{X}_b)$ is surjective. Then $\mathfrak{X}$ is complete at $b.$
\end{enumerate}

\noindent As a consequence, if $(\mathfrak{X}, \pi, B)$ is a deformation which is complete at a point $b$ of $B$ and such that $\mathrm{H}^{2}(\mathfrak{X}_b, \mathrm{T} \mathfrak{X}_b)=0,$ then~$\mathrm{KS}_b(\mathfrak{X})$ is surjective.

\begin{defi}\label{blowup}
Let $(\mathfrak{X}, \pi, B)$ be a deformation. The {\it blown up deformation} $\widehat{\mathfrak{X}}$ is a deformation over $\mathfrak{X}$ defined by~$\widehat{\mathfrak{X}}=\mathrm{Bl}_\mathfrak{X}(\mathfrak{X}\times_B\mathfrak{X}),$ where $\mathfrak{X}$ is diagonally embedded in $\mathfrak{X}\times_B\mathfrak{X}$ and the projection from $\widehat{\mathfrak{X}}$ to $\mathfrak{X}$ is induced by the projection on the first factor.
\end{defi}

\noindent Thus, for any $x$ in $\mathfrak{X},$ $\widehat{\mathfrak{X}}_x=\mathrm{Bl}_x( \mathfrak{X}_b),$ where $b=\pi(x).$
The following result is due originally to Fujiki and Nakano and in a more general setting to Horikawa:

\begin{pro} [\cite{FN, Ho}]\label{horik}
{\sl Let $(\mathfrak{X}, \pi, B)$ be a deformation, $b$ be a point of $B$ and assume that $\mathfrak{X}$ is complete at $b.$ Then the blown up deformation $\widehat{\mathfrak{X}}$ is complete at any point of $\, \mathfrak{X}_b.$}
\end{pro}

\noindent Remark that for every integer $N,$ $\widehat{\mathfrak{X}}_N=\mathfrak{X}_{N+1}.$ Since $\mathfrak{X}_0$ is complete, it follows by induction that for every integer $N,$ $\mathfrak{X}_N$ is complete at any point of $S_N.$

\begin{lem}\label{vanish}
{\sl Let $X$ be a rational surface obtained from the projective plane $\mathbb{P}^2$ via $N_{+}$ blow up and $N_{-}$ blow down. If $N=N_{+}-N_{-},$ then:
\begin{itemize}
\item $\mathrm{h}^1(X,\mathrm{T}X)=\mathrm{h}^0(X,\mathrm{T}X)+2N-8;$
\item $\mathrm{h}^2(X,\mathrm{T}X)=0.$
\end{itemize}}

\end{lem}

\begin{proof}
\emph{See} \cite[p. \!$220.$]{Kodaira}.
\end{proof}

\noindent We can now prove Theorem \ref{easy}.

\par \medskip

\begin{proof}[Proof of Theorem \ref{easy}]
Let $N$ be a positive integer and $\widehat{\xi}$ be a point in $S_N.$ The second statement of Lemma \ref{vanish} together with the completeness of $\mathfrak{X}_N$ implies that the Kodaira-Spencer map of $\mathfrak{X}_N$ is surjective at $\widehat{\xi}.$
Since the restriction of $\mathfrak{X}_{N}$ on $\mathcal{O}_{\widehat{\xi}}$ is trivial,~$\ker \mathrm{KS}_{\widehat{\xi}}( \mathfrak{X}_N)$ contains $\mathrm{T}_{\widehat{\xi}}\, \mathcal{O}_{\widehat{\xi}}.$ Let us compute the dimension of $\mathrm{T}_{\widehat{\xi}}\, \mathcal{O}_{\widehat{\xi}}.$ If $G_{\widehat{\xi}}$ is the stabilizer of $\widehat{\xi}$ in $\mathrm{PGL}(3;\mathbb{C}),$ one has an exact sequence $$0\longrightarrow\mathrm{Lie}(G_{\widehat{\xi}})\longrightarrow\mathrm{Lie}(\mathrm{PGL}(3;\mathbb{C}))\longrightarrow \mathrm{T}_{\widehat{\xi}}\, \mathcal{O}_{\widehat{\xi}}\longrightarrow 0.$$ Thus, by Lemma \ref{idcomp}, we obtain:
$\dim( \mathrm{T}_{\widehat{\xi}}\, \mathcal{O}_{\widehat{\xi}})=8-\mathrm{h}^0(X,\mathrm{T}X).$ Otherwise, since $\mathrm{KS}_{\widehat{\xi}}(\mathfrak{X}_N)$ is surjective, we get $$\dim (\ker \mathrm{KS}_{\widehat{\xi}}(\mathfrak{X}_N))=2N-\mathrm{h}^1(X, \mathrm{T}X)=8-\mathrm{h}^0(X, \mathrm{T}X)$$ by the first assertion of Lemma \ref{vanish}.
\end{proof}

\noindent Remark that if $N \geq 4,$ the kernels of the Kodaira-Spencer maps of $\mathfrak{X}_N$ define a holomorphic vector bundle of rank eight on $S_N^ {\dag},$ which is the tangent bundle of the regular foliation defined by the $\mathrm{PGL(3;\mathbb{C})}$-action on $S_N^ {\dag}.$

\subsection{Generic numbers of parameters of an algebraic deformation}\label{gennum}

In the section, we define the generic numbers of parameters of an algebraic deformation. Recall that a deformation $(\mathfrak{X}, \pi, B)$ is algebraic if there exists an embedding $i \colon \mathfrak{X}\rightarrow B \times \mathbb{P}^N$ such that $\pi$ is induced by the first projection of $B \times \mathbb{P}^N.$ If $\mathfrak{X}$ is algebraic, the fibers $(\mathfrak{X}_b)_{b \in B}$ are complex projective varieties. We always assume that $B$ is connected.

\begin{pro} \label{subbundle}
{\sl Let $(\mathfrak{X}, \pi, B)$ be an algebraic deformation. Then there exist a proper closed analytic subset $Z$ of $B$ and a holomorphic vector bundle $E$ on $U=B \setminus Z$ such that:
\begin{itemize}
\item E is a holomorphic subbundle of $\mathrm{T}U;$
\item the function $b \mapsto \mathrm{h}^1(\mathfrak{X}_b, \mathrm{T}\mathfrak{X}_b)$ is constant on $U;$
\item for all $b$ in $B,$ $E_{\vert b}$ is the kernel of $\mathrm{KS}_b(\mathfrak{X}).$
\end{itemize}}
\end{pro}
\begin{proof}
Let $\mathrm{T}^{\mathrm{rel}} \mathfrak{X}$ be the relative tangent bundle of $\mathfrak{X}$ defined by the exact sequence
$$0 \rightarrow \mathrm{T}^{\mathrm{rel}} \mathfrak{X} \rightarrow \mathrm{T} \mathfrak{X} \rightarrow \pi^{*} \mathrm{T}B \rightarrow 0,$$
where the last map is the differential of $\pi.$
The connection morphism $\mu\colon \mathrm{T}B \simeq \mathrm{R}^0 \pi_* (\pi^* \mathrm{T}B) \to
\mathrm{R}^1 \pi_* \mathrm{T}^{\mathrm{rel}} \mathfrak{X}$ induces for every $b$ in $B\,$ a map $$\mu_b \colon \mathrm{T}_b B \longrightarrow (\, \mathrm{R}^1 \pi_* \mathrm{T} \mathfrak{X}^{\mathrm{rel}})_{\vert b} \longrightarrow \mathrm{H}^1(\mathfrak{X}_b, \mathrm{T} \mathfrak{X}_b)$$ which is exactly the Kodaira-Spencer map of $\mathfrak{X}$ at $b$ (\emph{see} \cite[p. $219$]{Vo}). Since the deformation $\mathfrak{X}$ is algebraic, there exists a complex $\mathcal{E}^{\bullet}$ of vector bundles on $B$ such that for every~$b$ in $B,$ $\mathrm{H}^1(\mathfrak{X}_b, \mathrm{T} \mathfrak{X}_b)$ is the cohomology in degree one of the complex $\mathcal{E}_{\vert b}$ (\emph{see} \cite[p. $220$]{Vo}). This implies that the function $b \mapsto \dim \mathrm{H}^1(\mathfrak{X}_b, \mathrm{T} \mathfrak{X}_b)$ is constant outside a proper analytic subset $Z$ of $B.$ By Grauert's theorem \cite[p. $288$]{Ha}, $\mathrm{R}^1 \pi_* \mathrm{T} \mathfrak{X}^{\mathrm{rel}}$ is locally free on $U=B \setminus Z$ and for every~$b$ in $U,$ the base change morphism from $\mathrm{R}^1 \pi_* \mathrm{T} \mathfrak{X}^{\mathrm{rel}}_{\vert b}$ to $\mathrm{H}^1(\mathfrak{X}_b, \mathrm{T} \mathfrak{X}_b)$ is an isomorphism. After removing again a proper analytic subset in $U,$ we can assume that $\mu$ has constant rank on $U,$ so that its kernel is a holomorphic vector bundle.
\end{proof}

\noindent This being done, the definition of the generic number of parameters of an algebraic deformation runs as follows:

\begin{defi}  The number $\mathfrak{m}(\mathfrak{X})=\dim B-\mathrm{rank}\, E$ is called the \textit{generic number of parameters of} $\mathfrak{X}.$
\end{defi}

\begin{rem}
\begin{enumerate}
\item [(i)] Recall that a deformation $(\mathfrak{X}, \pi, B)$ is called \textit{effectively parameterized} (resp. \textit{generically effectively parameterized}) if for every $b$ in $B$ (resp. for every generic $b$ in B), the Kodaira-Spencer map $\mathrm{KS}_b(\mathfrak{X})$ is injective (\emph{see} \cite[p. 215]{Kodaira}). By Proposition \ref{subbundle}, an algebraic deformation $(\mathfrak{X}, \pi, B)$ is generically effectively parameterized if and only if $\mathfrak{m}(\mathfrak{X})=\dim B.$

\item [(ii)] By Theorem \ref{easy}, for any integer $N \geq 4,$ $\mathfrak{m}(\mathfrak{X}_N)=2N-8.$
\end{enumerate}
\end{rem}

\subsection{How to count parameters in a family of rational surfaces?}

\noindent Let $\mathfrak{Y}$ be a family of rational surfaces parameterized by an open set $U$ of $\mathbb{C}^n.$ Since the deformations $\mathfrak{X}_N$ are complete, we can suppose that $\mathfrak{Y}$ is obtained by pulling back the deformation $\mathfrak{X}_N$ by a  holomorphic map $\psi \colon U \rightarrow S_N.$ We will make the assumption that the fibers of $\mathfrak{Y}$ have no holomorphic vector field, so that $\psi$ takes its values in $S_N^{\dag}.$ In this situation, we are able to compute the numbers of parameters of such a family quite simply:

\begin{thm}\label{critere}
{\sl Let $U$ be an open set in $\mathbb{C}^n,$ $N$ be an integer greater than or equal to $4$ and $\psi\colon U \to S_N^{\dag}$ be a holomorphic map. Then $\mathfrak{m}(\psi^* \mathfrak{X}_N)$ is the smallest integer $k$ such that for all generic $\alpha$ in $U,$ there exist a neighborhood $\Omega$ of $0$ in~$\mathbb{C}^{n-k}$ and two holomorphic maps $\gamma \colon \Omega \rightarrow U$ and $M \colon \Omega \rightarrow \mathrm{PGL(3; \mathbb{C})}$ such that:
\begin{itemize}
\item $\gamma_{*}(0\!)$ is injective,
\item $\gamma(0\! )=\alpha$ and $M(0\!)=\mathrm{Id},$
\item for all $t$ in $\Omega,$ $\psi \, (\gamma (\, t \!))=M(\, t)\,  \psi(\alpha).$
\end{itemize}}
\end{thm}

\begin{proof}
\noindent Let $\alpha$ be a generic point in $U,$ $U_{\alpha}$ be a small neighborhood of $\alpha$ and $Z_{\alpha}=\psi(U_{\alpha});$ $Z_{\alpha}$ is a smooth complex submanifold of $S_N^{\dag}$ passing through $\psi(\alpha).$ The rank of $\psi$ is generically constant, so that after a holomorphic change of coordinates, we can suppose that $U_{\alpha}=V_{\alpha} \times Z_{\alpha}$ and that $\psi$ is the projection on the second factor. If $(v,z)$ is a point of~$V_{\alpha} \times Z_{\alpha},$ the kernel of $\mathrm{KS}_{(v,z)}(\psi^* \mathfrak{X}_N)$ is the set of vectors $(h,k)$ in $\mathrm{T}_v V_{\alpha} \oplus \mathrm{T}_z Z_{\alpha}$ such that $k$ is tangent to the orbit~$O_z.$ If~$\alpha$ is sufficiently generic, these kernels define a holomorphic subbundle of $\mathrm{T}(V_{\alpha} \times Z_\alpha)$ of rank $n-\mathfrak{m}(\psi^* \mathfrak{X}_N),$ which is obviously integrable because the $\mathrm{PGL(3;\mathbb{C})}$-orbits in $S_N^{\dag}$ define a regular foliation. Let $V_{\alpha} \times {T}_\alpha$ be the associated germ of integral manifold passing through $\alpha.$ For every point $z$ in $T_\alpha,$ $\mathrm{T}_z T_\alpha$ is included in $\mathrm{T}_z O_z.$ Thus $T_\alpha$ is completely included in the orbit $O_{\psi(\alpha)}.$ Let $\gamma$ be a local parametrization of $V_{\alpha} \times T_\alpha.$ As the natural orbit map from $\mathrm{PGL(3; \mathbb{C})}$ to $\mathcal{O}_{\psi(\alpha)}$ is a holomorphic submersion, we can choose locally around $\psi(\alpha )$ a holomorphic section $\tau$ such that $\tau(\psi(\alpha))=\mathrm{Id}.$ If we define $M(\, t)=\tau[\gamma(t \!)],$ then $\gamma(\, t\!)=M(\, t) \, \psi(\alpha).$

\par \medskip

\noindent Conversely, let $\alpha$ be a generic point in $U,$ $d$ be an integer and $(\gamma,\, M)$ satisfying the assumptions of the theorem. The image of~$\gamma \,$ defines a germ of smooth subvariety $Y_{\alpha}$ in $U$ passing through $\alpha,$ and its image by $\psi$ is entirely contained in the orbit~$O_{\psi(\alpha)}.$ This implies that the restriction of $\psi^* (\mathfrak{X}_N)$ to $Y_{\alpha}$ is holomorphically trivial. Thus $\mathrm{T}_{\alpha}Y_{\alpha}$ is contained in the kernel of $\mathrm{KS}_{\alpha}(\psi^* \mathfrak{X}_N).$ Since $\mathrm{dim}\, Y_{\alpha}=n-k,$ we obtain the inequality $\mathfrak{m}(\psi^* \mathfrak{X}_N)\leq k.$
\end{proof}

\begin{eg}
Let us consider the family of birational maps $(\varphi_{\alpha}f)_{\alpha \in \mathbb{C}^*}$ defined in Theorem \ref{conique}. This family can be lifted to a family of rational surface automorphisms. In the notations of Proposition \ref{coniquecomp}; the associated deformation of rational surfaces is $\psi^* \mathfrak{X}_{15},$ where $\psi \colon \mathbb{C}^* \rightarrow S_{15}^{\dag}$ is given by $\psi(\alpha)=\big(\,\widehat{\zeta}_1, P,\, \varphi_{\alpha} (\widehat{\zeta}_2),\, \varphi_{\alpha}(R),\, \varphi_{\alpha} f \varphi_{\alpha} (\widehat{\zeta}_2),\, \varphi_{\alpha} f \varphi_{\alpha} (R) \big).$ For any point $\alpha_0$ in $\mathbb{C}^*,$ let $\Omega$ be a small neighborhood of $0$ in $\mathbb{C}$ and $\alpha \colon \Omega \rightarrow \mathbb{C}^*,$ $M \colon \Omega \rightarrow \mathrm{PGL}(3; \mathbb{C})$ be two holomorphic maps such that $\alpha(0\!)=\alpha_0,$ $M(0\!)=\mathrm{id}$ and for all $t$ in $\Omega,$ $\psi(\alpha(t\!))= M(t\!).\psi(\alpha_0).$ This means that:
\begin{itemize}
\item[(i)] $\widehat{\zeta}_1$ and $P$ are fixed by $M(t\!),$
\par \smallskip

\item[(ii)] $\widehat{\zeta}_2$ and $R$ are fixed by $\varphi_{\alpha (t\!)}^{-1} M(t\!)\varphi_{\alpha_0},$

\item[(iii)] $\widehat{\zeta}_2$ and $R$ are fixed by $(\varphi_{\alpha (t\!)} f \varphi_{\alpha (t\!)})^{-1} M(t\!)\varphi_{\alpha_0} f \varphi_{\alpha_0}$.
\end{itemize}

\noindent The stabilizer of $\widehat{\zeta}_1$ in $\mathrm{PGL}(3; \mathbb{C})$ consists of matrices of the form $\left[
\begin{array}{ccc} 1 & 0 & 0\\
0 & w & 0 \\
0 & 0 & w^2
\end{array}
\right],$ $w\in\mathbb{C}^*.$ These matrices also fix the point $P.$
Thus, condition (i) implies that $M(t\!)=\left[
\begin{array}{ccc} 1 & 0 & 0\\
0 & A(t\!) & 0 \\
0 & 0 & A(t\!)^2
\end{array}
\right],$ where $A \colon \Omega \rightarrow \mathbb{C}^*$ is a holomorphic map such that $A(0\!)=1.$ For (ii) and (iii), we must compute conditions analogous to those of Proposition \ref{recolconfig9} to describe germs of bihilomorphisms $u$ such that $u(\widehat{\zeta}_2)=\widehat{\zeta}_2.$ If we take the power series expansion of $u$ in the coordinates $(x,z),$ one checks that the conditions are the following ones:
\begin{itemize}
\item $m_{0,0}=n_{0,0}=0,$
\item $m_{0,1}=0,$
\item $m_{0,2}+m_{1,0}-n_{0,1}^2=0,$
\item $m_{0,3}+m_{1,1}-2n_{0,1}(n_{0,2}+n_{1,0})=0.$
\end{itemize}

\noindent We compute the Taylor expansion of $\varphi_{\alpha (t\!)}^{-1} M(t\!)\varphi_{\alpha_0}$ at $Q$ and explicit the above conditions. They yield that (ii) is satisfied if and only if $A(t\!)=\frac{\alpha(t\!)}{\alpha_0}.$ Then another computation shows that (iii) is always satisfied (\textit{i.e.} (iii) imposes no further restriction on the function $\alpha$), so that $\mathfrak{m}(\psi^* \mathfrak{X}_{15})=0.$ This fact will also be a consequence of Theorem \ref{genius}.
\end{eg}

\subsection{Nonbasic rational surfaces}\label{nonbasic}
We will briefly explain how to adapt the methods developed above to nonbasic rational surfaces, although we won't need it in the paper. The situation is more subtle, even for Hirzebruch surfaces. Indeed, if $n\geq 2,$ $\mathrm{Aut}(\mathbb{F}_n)$ has dimension $n+5$ (\emph{see} \cite{Bea}) so that $\mathrm{h}^1(\mathbb{F}_n, \mathrm{T}\mathbb{F}_n)=n-1$ and $\mathrm{h}^2(\mathbb{F}_n, \mathrm{T}\mathbb{F}_n)=0$ by Lemma~\ref{vanish}. Therefore the Hirzebruch surfaces $\mathbb{F}_n$ are not rigid if $n \geq 2.$ Complete deformations of Hirzebruch surfaces $(\mathbb{F}_n)_{n \geq 2}$ are known and come from flat deformations of rank-two holomorphic bundles on $\mathbb{P}^1(\mathbb{C})$ (\emph{see} \cite[Chap.~II]{Manetti}). These deformations are highly non-effectively parameterized because their generic number of parameters is zero. We will denote them by $(\mathfrak{F}_n, U_n),$ where $U_n$ is a neighborhood of the origin in $\mathbb{C}^{n-1}$ and $(\mathfrak{F}_n)_0=\mathbb{F}_n.$ The fibers of~$\mathfrak{F}_n$ over points of $U_n \setminus \{0\}$ are Hirzebruch surfaces $\mathbb{F}_{n-2k}$ of smaller index.
\par \medskip

\noindent The deformations of nonbasic rational surfaces can be explicitly described using the same method as in \S\ref{def}: for every integer $n \geq 2,$ let us define inductively a sequence of deformations $\tilde{\pi}_{N,n}\colon\mathfrak{F}_{N,n} \to S_{N,n}$ by $\mathfrak{F}_{0,n}=\mathfrak{F}_n$ and $\mathfrak{F}_{N+1,n}=\widehat{\mathfrak{F}_{N,n}}$ (\emph{cf} Definition~\ref{blowup}). This means that
$$S_{N,n}=\{a,\, p_1, \ldots , p_N \, \vert \, a \in U_n,\, p_1 \in (\mathfrak{F}_n)_a,\, p_2 \in \mathrm{Bl}_{p_1}(\mathfrak{F}_n)_a, \ldots, p_N \in \mathrm{Bl}_{p_{N-1}} \ldots \mathrm{Bl}_{p_1}(\mathfrak{F}_n)_a\}$$
and that $(\mathfrak{F}_{N,n})_{a,\, p_1, \ldots, p_N}= \mathrm{Bl}_{p_N} \ldots \mathrm{Bl}_{p_1}(\mathfrak{F}_n)_a.$

\noindent If $X=\mathrm{Bl}_{\widehat{\xi}\, } \mathbb{F}_n$ is a nonbasic rational surface, then $\widehat{\xi}$ defines a point in $S_{N,n}$ for a certain integer $N.$
By Proposition \ref{horik}, $\mathfrak{F}_{N,n}$ is complete at $\widehat{\xi}.$ Therefore small deformations of a nonbasic rational surface can be parameterized by (possibly infinitely near) points on Hirzebruch surfaces $\mathbb{F}_n,$ but $n$ can jump with the deformation parameters.

\subsection{Application to families of Cremona transformations}\label{nonbasic}
The aim of this section is to relate to different notions of "generic number of parameters", the first one being introduced in \S \ref{desing} for holomorphic families of birational maps and the second one in \S\ref{gennum} for arbitrary algebraic deformations. Our first main result is:

\begin{thm}\label{inequality}
{\sl Let $N$ and $d$ be positive integers such that $N$ is greater than or equal to $4,$ $Y$ be a smooth connected analytic subset of $\mathrm{Bir}_d(\mathbb{P}^2)$ and $\psi \colon Y \rightarrow S_N^{\dag}$ be a holomorphic map. If $\mathfrak{X}=\psi^{*} \mathfrak{X}_N,$ let $\Gamma \colon \mathfrak{X} \rightarrow Y \times \mathbb{P}^2$ be the natural holomorphic map over $Y$ whose restriction on each fiber $\mathfrak{X}_y$ is  the natural projection from $\mathrm{Bl}_{\psi(y)} \mathbb{P}^2$ to~$\mathbb{P}^2.$ Assume that for any $y$ in $Y,$ if $f_y$ is the birational map parameterized by $y,$ ${\Gamma_y}^{-1} \circ f_y \circ \Gamma_y$ is an automorphism of the rational surface $\mathfrak{X}_y.$
Then the generic number of parameters of the holomorphic family $Y$ is smaller than the generic number of parameters of the deformation $\mathfrak{X},$ {\it i.e.} $\mathfrak{m}(Y)\leq \mathfrak{m}(\mathfrak{X}).$}
\end{thm}

\begin{proof}
Let $y$ be a generic point in $Y.$ By Theorem \ref{critere}, we can find a complex submanifold $\Omega$ of $Y$ of codimen\-sion~$\mathfrak{m}(\mathfrak{X})$ passing through $y$ as well as a holomorphic map $M \colon \Omega \rightarrow \mathrm{PGL}(3; \mathbb{C})$ such that $M(y)=\mathrm{id}$ and for every $t$ in~$\Omega,$ $\psi(t)=M(t) \psi(y).$  Let $\Delta \colon \Omega \times \mathfrak{X}_y \rightarrow \mathfrak{X}_{\vert\Omega}$ be the associated global holomorphic trivialization of the deformation~$\mathfrak{X}_{\vert\Omega}:$ for every $t$ in $\Omega,$ $\Delta_{t}$
is the isomorphism between $\mathrm{Bl}_{\psi(y)}\mathbb{P}^2$ and $\mathrm{Bl}_{\psi(t)}\mathbb{P}^2$ induced by $M(t).$ This implies that $$\Delta^{-1}_{t} \circ \Gamma ^{-1}_{t} \circ f_{t} \circ \Gamma_{t} \circ \Delta_{t}=\Gamma^{-1}_y \circ [M(t)^{-1} \circ f_{t} \circ M(t)] \circ \Gamma_{y}.$$ Since $\mathfrak{X}_y$ has no nontrivial holomorphic vector field, the holomorphic family $(\Gamma^{-1}_y \circ [M(t)^{-1} \circ f_{t} \circ M(t)] \circ \Gamma_{y})_{t \in \Omega}$ of automorphisms of $\mathfrak{X}_y$ must be constant. Thus, we obtain that for every $t$ in $\Omega,$ $f_{t}=M(t) \circ f_y \circ M(t)^{-1}.$ This means that $\Omega$ is contained in the adjoint orbit $O_y$ of $f_y,$ so that $\mathfrak{m}(Y)= \mathrm{codim} \, (Y \cap O_y)^{irr} \leq \mathrm{codim} \, \Omega=\mathfrak{m}(\mathfrak{X}).$
\end{proof}

\noindent It is easy to produce examples where $\mathfrak{m}(Y) < \mathfrak{m}(\mathfrak{X}):$ let $d=1,$ $\iota$ be a linear involution which is not in the center of~$\mathrm{PGL}(3; \mathbb{C})$ and $Y$ be a smooth curve in the adjoint orbit $O_{\iota}$ passing through $\iota.$ We can assume that there exists a holomorphic map $M \colon Y \rightarrow \mathrm{PGL}(3; \mathbb{C})$ such that $M(\iota)=\mathrm{id}$ and for all $y$ in $Y,$  $f_y=M(y)\, \iota \, M(y)^{-1}.$ Let us choose four distinct generic points $p_1,$ $p_2,$ $p_3,$ and $p_4$ in $\mathbb{P}^2$ such that for $1 \leq i,j \leq 4,$ $\iota(p_i) \neq p_j.$ We define a holomorphic function $\psi \colon Y \rightarrow S_8$ by the formula $$\psi(y)=\big(M(y)(p_1),\, M(y)(p_2),\, M(y)(p_3),\, M(y)(p(y)),\, M(y)(\iota(p_1)),\, M(y)(\iota(p_2)),\, M(y)(\iota(p_3)),\, M(y)(\iota[p(y)])\big),$$
\noindent where $p \colon Y \rightarrow \mathbb{P}^2$ is a holomorphic immersion such that~$p(\iota )=p_4.$
Since the points $p_i$ are generic, we can assume that $\psi$ takes its values in $S_8^{\dag}.$ Besides, for any $y$ in $Y,$ the involution $f_y$ can be lifted to an automorphism of $\mathrm{Bl}_{\psi(y\, )}\mathbb{P}^2.$ Theorem \ref{critere} implies that $\mathfrak{m}(\psi^* \mathfrak{X}_8)=1,$ but $\mathfrak{m}(Y)=0.$
\par \medskip

\begin{thm}\label{genius}
{\sl Let $k$ and $N$ be two positive integers, $f$ be a birational map of the complex projective plane, $\widehat{\xi}_1,$ and $\widehat{\xi}_2$ be two points of $S_N$ corresponding to the minimal desingularization of $f$ and $U$ be a smooth connected analytic subset of~$\mathrm{PGL}(3; \mathbb{C}).$ We make the following assumptions:

\begin{itemize}
\item [(i)] For all $\varphi$ in $U,$ $(\varphi f)^k \varphi\,  \widehat{\xi}_2=\widehat{\xi}_1.$

\item [(ii)] The supports of $\widehat{\xi}_1,$ $\varphi \widehat{\xi}_2$ and $(\varphi f)^j \varphi\,  \widehat{\xi}_2, \, \, 1 \leq j \leq k-1,$ are pairwise disjoint.

\item [(iii)] If $\psi \colon U \rightarrow S_{kN}$ is defined by $\psi(\varphi)=( \widehat{\xi}_1, \varphi \, \widehat{\xi}_2, \varphi f  \varphi\,  \widehat{\xi}_2, \ldots, (\varphi f)^{k-1} \varphi\,  \widehat{\xi}_2),$ then the image of $\psi$ is included in $S_{kN}^{\dag}.$

\item[(iv)] For all $\varphi$ in $U,$ the birational map $\varphi f$ can be lifted to an automorphism of the rational surface $\mathrm{Bl}_{\psi(\varphi)}\mathbb{P}^2.$
\end{itemize}

\noindent If $\widetilde{U}$ denotes the family of birational maps $(\varphi f)_{\varphi \in U}$ and if $\mathfrak{X}=\psi^*\mathfrak{X}_{kN},$ then $\mathfrak{m}(\widetilde{U}\!)=\mathfrak{m}(\mathfrak{X}).$}
\end{thm}

\begin{proof}
By Theorem \ref{inequality}, we know that $\mathfrak{m}(\widetilde{U}\!)\leq \mathfrak{m}(\mathfrak{X}).$ To prove the converse inequality, let us choose a generic point $\varphi $ in $U.$ Then the intersection $Z$ of $\widetilde{U}$ with the adjoint orbit $O_{\varphi f}$ of $\varphi f$ is smooth of codimension $\mathfrak{m}(\widetilde{U})$ in a neighborhood of $\varphi f.$  If $v \colon \mathrm{PGL}(3; \mathbb{C}) \rightarrow O_{\varphi f}$ is the orbit map associated with the adjoint action of $\mathrm{PGL}(3; \mathbb{C}),$ then $v$ is a holomorphic submersion. Thus we can choose locally a holomorphic section $M \colon Z \rightarrow \mathrm{PGL}(3; \mathbb{C})$ of $v$ near $\varphi f$ such that $M(\varphi f)=\mathrm{id}.$ For every $z$ in $Z,$ if $\varphi_z f$ is the birational map corresponding to $z,$ then $\varphi_z f=M(z)\, \varphi f M(z)^{-1}.$ We use now the essential assumption: $\widehat{\xi}_1$ and~$\widehat{\xi}_2$ correspond to a \textit{minimal} desingularization of $f.$ By Lemma \ref{minimal}, we obtain that for any $z$ in $Z,$ $M(z)\widehat{\xi}_1=\widehat{\xi}_1$ and $M(z)\varphi\widehat{\xi}_2=\varphi_z \widehat{\xi}_2.$ Thus, $\psi(z)=M(z)\psi(\varphi f)$ and this implies by Theorem \ref{critere} that $\mathfrak{m}(\mathfrak{X})\leq \mathrm{codim}\, Z=\mathfrak{m}(\widetilde{U}).$
\end{proof}

\vspace*{8mm}

\bibliographystyle{alpha}
\bibliography{indiana}
\nocite{*}

\end{document}